\documentclass[reqno,twoside]{amsart} 

\usepackage{amsmath,amsthm,amssymb,amsfonts,mathrsfs,color}
\usepackage{latexsym,esint}
\usepackage{tikz}
\usetikzlibrary{calc}

\theoremstyle{plain}
\begingroup
\newtheorem{thm}{Theorem}[section]
\newtheorem{lem}[thm]{Lemma}
\newtheorem{prop}[thm]{Proposition}
\newtheorem{cor}[thm]{Corollary}
\endgroup

\theoremstyle{definition}
\begingroup
\newtheorem{defn}[thm]{Definition}
\newtheorem{rem}[thm]{Remark}
\endgroup

\numberwithin{equation}{section}

\makeatletter
\def\rightharpoonupfill@{\arrowfill@\relbar\relbar\rightharpoonup}
\newcommand{\xrightharpoonup}[2][]{\ext@arrow 0359\rightharpoonupfill@{#1}{#2}}
\makeatother

\newcommand{\ds}{\displaystyle}

\renewcommand{\Cap}{{\rm Cap}}
\newcommand{\CC}{\mathbb C} 
 
\newcommand{\R}{\mathbb R} 
\newcommand{\Rn}{{\mathbb R}^n} 
\newcommand{\Ms}{{\mathbb M}^{2{\times}2}_{\rm sym}}

\newcommand{\dist}{{\rm dist}}

\renewcommand{\div}{{\rm div}\,}
\newcommand{\wto}{\rightharpoonup}
\newcommand{\e}{\varepsilon}

\newcommand{\A}{{\mathcal A}}
\newcommand{\AR}{{\mathcal A}_{r}}

\newcommand{\Q}{{\mathcal Q}}
\newcommand{\LL}{{\mathcal L}}
\newcommand{\HH}{{\mathcal H}}
\newcommand{\M}{{\mathcal M}}
\newcommand{\C}{{\mathcal C}}

\newcommand{\K}{{\mathcal K}}

\newcommand{\Dc}{\beta}

\let\O=\Omega

\begin{document}
 
\title[Stress regularity in quasi-static perfect plasticity]{Stress regularity in quasi-static perfect plasticity with a pressure dependent yield criterion}
\author[J.-F. Babadjian]{Jean-Fran\c cois Babadjian}
\author[M.G. Mora]{Maria Giovanna Mora}

\address[J.-F. Babadjian]{Sorbonne Universit\'es, UPMC Univ Paris 06, CNRS, UMR 7598, Laboratoire Jacques-Louis Lions, F-75005, Paris, France}
\email{jean-francois.babadjian@upmc.fr}

\address[M.G. Mora]{Dipartimento di Matematica, Universit\`a di Pavia, via Ferrata 1, 27100 Pavia, Italy}
\email{mariagiovanna.mora@unipv.it}

\date{\today}

\subjclass{}
\keywords{Elasto-plasticity, Convex analysis, Quasi-static evolution, Regularity, Functions of bounded deformation, Capacity}

\begin{abstract}
This work is devoted to establishing a regularity result for the stress tensor in quasi-static planar isotropic linearly elastic -- perfectly plastic materials obeying a Drucker-Prager or Mohr-Coulomb yield criterion. Under suitable assumptions on the data, it is proved that the stress tensor has a spatial gradient that is locally squared integrable. As a corollary, the usual measure theoretical flow rule is expressed in a strong form using the quasi-continuous representative of the stress.
\end{abstract}

\maketitle


\section{Introduction}

Perfect plasticity is a class of models in continuum solid mechanics involving a fixed threshold criterion on the Cauchy stress. When the stress is below a critical value, the underlying material behaves elastically, while the saturation of the constraint leads to permanent deformations after unloading back to a stress-free configuration. Elasto-plasticity represents a typical inelastic behavior, whose evolution is described by means of an internal variable, the plastic strain.\medskip

To formulate more precisely the problem, let us consider a bounded open set $\O \subset \R^n$ (in the following, only the dimension $n=2$ will be considered), 
which stands for the reference configuration of an elasto-plastic body. In the framework of small strain elasto-plasticity the natural kinematic and static variables are the displacement field $u :\O \times [0,T] \to \Rn$ and the stress tensor $\sigma:\O \times [0,T] \to {\mathbb M}^{n{\times}n}_{\rm sym}$, where  ${\mathbb M}^{n{\times}n}_{\rm sym}$ is the set of $n \times n$ symmetric matrices. In quasi-statics the equilibrium is described by the system of equations
$$
- \div\sigma= f \quad \text{in }\O \times [0,T],
$$
for some given body loads $f:\O \times [0,T] \to \R^n$. Perfect plasticity is characterized by the existence of a yield zone in  which the stress is constrained to remain. The stress tensor must indeed belong to a given closed and convex subset $K$ of ${\mathbb M}^{n{\times}n}_{\rm sym}$ with non empty interior:
$$
\sigma \in K.
$$
If $\sigma$ lies inside the interior of $K$, the material behaves elastically, so that unloading will bring the body back to its initial configuration. On the other hand, if $\sigma$ reaches the boundary of $K$ (called the yield surface), a plastic flow may develop, so that, after unloading, a non-trivial permanent plastic strain will remain. The total linearized strain, denoted by  $Eu:=(Du+Du^T)/2$, is thus additively decomposed as
$$Eu=e+p.$$
The elastic strain $e:\O \times [0,T] \to {\mathbb M}^{n{\times}n}_{\rm sym}$ is  related to the stress through the usual Hooke's law 
$$\sigma:=\CC e,$$
where $\CC$ is the symmetric fourth order elasticity tensor. The evolution of the plastic strain $p:\O \times [0,T] \to {\mathbb M}^{n{\times}n}_{\rm sym}$ is described by means of the flow rule
\begin{equation}\label{eq:diff-inc}
\dot p \in N_K(\sigma),
\end{equation}
where $N_K(\sigma)$ is the normal cone to $K$ at $\sigma$. From convex analysis, $N_K(\sigma)=\partial I_K(\sigma)$, {\it i.e.}, 
it coincides with the subdifferential of the indicator function $I_K$ of the set $K$ (where $I_K(\sigma)=0$ if $\sigma \in K$, while $I_K(\sigma)=+\infty$ otherwise). Hence, from convex duality, the flow rule can be equivalently written as
\begin{equation}\label{eq:Hill}
\sigma{\,:\,}\dot p=\max_{\tau \in K}\tau{\,:\,}\dot p=:H(\dot p),
\end{equation}
where $H :{\mathbb M}^{n{\times}n}_{\rm sym} \to [0,+\infty]$ is the support function of $K$. This last formulation \eqref{eq:Hill} is nothing but Hill's principle of maximum plastic work, and $H(\dot p)$ denotes the plastic dissipation.\medskip

Standard models used for most of metals or alloys are those of Von Mises and Tresca. These kinds of materials are not sensitive to hydrostatic pressure, and plastic behavior is only generated  through critical shearing stresses. In these models, if $\sigma_D:=\sigma-\frac{{\rm tr}\sigma}{n} {\rm Id}$ stands for the deviatoric stress, the elasticity set $K$ is of the form
$$\{\sigma \in {\mathbb M}^{n{\times}n}_{\rm sym}: \kappa(\sigma_D) \leq k\},$$
where $k>0$. The Von Mises yield criterion corresponds to $\kappa(\sigma_D)=|\sigma_D|$, while that of Tresca to $\kappa(\sigma_D)=\sigma_n - \sigma_1$ where $\sigma_1 \leq \cdots \leq \sigma_n$ are the ordered principal stresses. The mathematical analysis of such models has been performed in \cite{Suquet,T,A,DMDSM}.

On the other hand, in the context of soil mechanics, materials such as sand or concrete turn out to develop permanent volumetric changes due to hydrostatic pressure. Typical models are those of Drucker-Prager and Mohr-Coulomb (see \cite{DP}), which can be seen as generalizations of the Von Mises and Tresca models, respectively. In these cases, the elasticity set $K$ takes the form
$$\{\sigma\in {\mathbb M}^{n{\times}n}_{\rm sym}: \kappa(\sigma_D) + \alpha {\rm tr}\sigma \leq k\},$$
where $\alpha,k>0$. The main difference between metals and soils, is that for the latter, there are in general no directions along which the Cauchy stress is bounded. These models have been studied in \cite{BM} (see also \cite{ReSe}).\medskip

A common feature to all the models of perfect plasticity described so far is that  they develop strain concentration leading to discontinuities of the displacement field. This has been a major difficulty in defining a suitable functional framework for the study of such problems. It has been overcome by the introduction of the space $BD$ of functions of bounded deformation (see \cite{Suquet3,T})
and through a suitable relaxation procedure (see \cite{AG, Mo}). Solutions in the energy space must at least satisfy the following regularity: for all $t \in [0,T]$,
$$u(t) \in BD(\O), \quad e(t),\; \sigma(t) \in L^2(\O;{\mathbb M}^{n{\times}n}_{\rm sym}), \quad p(t) \in \M(\overline \O;{\mathbb M}^{n{\times}n}_{\rm sym}),$$
where $\M(\overline \O;{\mathbb M}^{n{\times}n}_{\rm sym})$ stands for the space of bounded Radon measures in $\overline \O$. 

Higher regularity of solutions appears therefore as a natural question. For dynamical problems it has (only recently) been established in \cite{M} that for any elasticity set $K$ the solutions are smooth in short time, provided the data are smooth and compactly supported in space. Such a result does not hold in the static or quasi-static cases (see the examples in \cite[Section~2]{Suquet} or \cite[Section~10]{Dem}). However, some partial regularity results are available for the stress in some particular situations. Indeed, it has been proved in \cite{BF,S1,Dem,DM} that for a Von Mises elasticity set, the Cauchy stress satisfies
$$\sigma \in L^\infty(0,T;H^1_{\rm loc}(\O; {\mathbb M}^{n{\times}n}_{\rm sym})).$$
Unfortunately, the proofs of these results are very rigid and do not easily extend to other types of elasticity sets. A more general result has been obtained in \cite{GM}, where it has been proved that the same regularity result holds if $K=K_D \oplus \R {\rm Id}$, where $K_D$ is a smooth compact convex subset of deviatoric symmetric matrices, with positive curvature.\medskip

In the footsteps of \cite{BM}, we address here the question of deriving similar regularity properties for the stress tensor in the case of Drucker-Prager and Mohr-Coulomb elasticity sets. A result in this direction has been obtained in \cite{Se} in the static case. The object of this present work is to extend this result to the quasi-static case. We indeed prove that in dimension $n=2$, the stress tensor has the expected regularity $\sigma \in L^\infty(0,T;H^1_{\rm loc}(\O; \Ms))$ (see Theorem~\ref{thm:main}). 

Our proof rests on a duality approach, analogous to that of \cite{Se} and \cite{Dem}. Since the solutions to the perfect plasticity model are singular, one needs first to regularize the problem. To this aim we consider a visco-plastic approximation of Perzyna type (see \cite{Suquet}). In contrast with the Kelvin-Voigt visco-elastic regularization used in \cite{Se,Dem}, it modifies the dissipation potential $H$ (see \eqref{eq:Hill}) in such a way that the regularized flow rule is not given by a differential inclusion as in \eqref{eq:diff-inc}, but is a differential equation. In other words, the regularized plastic strain rate is univocally determined by the stress (see \eqref{eq:reg-flow-rule3}). To the best of our knowledge, this is the first time that this kind of approximation is used to establish regularity of the stress. 

In the following we explain the strategy of our proof and the main difficulties. For simplicity of exposition we neglect the terms due to the visco-plastic approximation (which are those $\e$-dependent), we assume the solution to be smooth,
and we show how a uniform estimate on the $H^1_{\rm loc}(\O; {\mathbb M}^{n{\times}n}_{\rm sym})$-norm of $\sigma(t)$ can be obtained.

In the static case \cite{Se} and in the absence of external body loads, the equilibrium states $(u,e,p)$ minimize the energy functional
$$(v,\eta,q) \mapsto \frac12 \int_\O \CC \eta:\eta\, dx + \int_\O H(q)\, dx$$
among all triples $(v,\eta,q)$ satisfying the additive decomposition $Ev=\eta+q$ and the boundary condition. Minimizing first with respect to $q$ shows that $u$ is actually a minimizer of
$$v \mapsto \int_\O g(Ev)\, dx,$$
where $g$ is the convex conjugate of the auxiliary energy $\tau \mapsto \frac12 \CC^{-1} \tau:\tau + I_K(\tau)$, and the stress is then given by 
\begin{equation}\label{eq:sigma=Dg}
\sigma=Dg(Eu).
\end{equation}
This formula for the stress is the starting point of the analysis. In order to get estimates on the spatial gradient of $\sigma$, it would be convenient to differentiate \eqref{eq:sigma=Dg}. Unfortunately, the function $g$ being only of class $\C^1$ with Lipschitz continuous partial derivatives, the classical chain rule formula for the composition of a Lipschitz function with a vector valued (Sobolev) function does not apply. To overcome this problem, we compute explicitely the expression of $g$ (see \eqref{eq:g}) and use a general chain rule formula established in \cite{ADM} to get a formula of the type $\partial_k \sigma=D^2g(Ev)\partial_k Ev$.

The study of the quasi-static case introduces further difficulties. In the case of Von Mises plasticity \cite{Dem}, the previous method is applied to the incremental problem where $(u_i,e_i,p_i)$ minimizes
$$(v,\eta,q) \mapsto \frac12 \int_\O \CC \eta:\eta\, dx + \int_\O H(q-p_{i-1})\, dx$$
among all admissible triples $(v,\eta,q)$ at time $t_i$, or still $u_i$ minimizes
$$v \mapsto \int_\O g(Ev-p_{i-1})\, dx.$$
It is shown that the stress $\sigma_i:=Dg(Eu_i-p_{i-1})$ satisfies an $H^1_{\rm loc}(\O;{\mathbb M}^{n{\times}n}_{\rm sym})$ estimate that is uniform with respect to the viscosity parameter, but may possibly depend on the time step. Afterwards, a uniform estimate with respect to the time step is established showing the desired regularity result. The main drawback of this approach is that it necessitates to perform twice almost the same computations. 

In contrast with \cite{Dem}, we directly work on a time continuous model, and use the underlying variational structure to establish a similar formula (see \eqref{form-sigma}) for the stress as 
$$\sigma=Dg(\xi), \quad \xi:=e+\dot p.$$ 
The strategy consists then in differentiating the equilibrium equation, take $(\partial_k \dot u)\psi$ as test function (where $\psi$ is a suitable cut-off function), and deduce an inequality that provides a bound on $\partial_k \sigma$ (see Proposition~\ref{prop:unif-bound}). The main difficulty in doing so is to deal with a term of the form (see \eqref{eq:stim6})
\begin{equation}\label{eq:term1}
\int_0^t \int_\O |\partial_k \sigma| |\xi|\psi\, dx \, ds,
\end{equation}
since $\xi$ is only bounded in $L^1(\O \times (0,T);{\mathbb M}^{n{\times}n}_{\rm sym})$. We thus need to absorb this term by some of the coercivity terms of the left hand side 
$$\int_0^t \int_\O \partial_k \sigma:\partial_k \sigma \psi\, dx \, ds + \int_0^t \int_\O \partial_k \sigma : \partial_k \xi \psi \, dx \, ds.$$
To this end, we use the special structure of the function $g$, together with the formula $\partial_k\sigma=D^2g(\xi) \partial_k \xi$, to show that the integral in \eqref{eq:term1} can be controled by 
$$M\left(\int_0^t \int_\O \partial_k \sigma : \partial_k \xi \psi \, dx \, ds\right)^{1/2},$$
where $M$ only depends on various norms of $(u,e,p)$ in the energy space (see \eqref{eq:estim71}--\eqref{eq:estim72}). 

Note that the estimates performed in the proof of Proposition~\ref{prop:unif-bound} should actually hold true in any space dimension. However, since the final estimate \eqref{eq:endly} involves the $L^2(\O \times (0,T);\R^n)$-norm of $\dot u$ and this is controled only in dimension $n=2$ owing to the continuous embedding of $BD(\O)$ into $L^2(\O;\R^2)$\footnote{In dimension $n=3$, $BD(\O)$ only embeds into $L^{3/2}(\O;\R^3)$}, this estimate turns out to be uniform with respect to the viscosity parameter only in dimension $n=2$. This special role played by the dimension $n=2$ was already observed in the papers \cite{BM,ReSe,Se}, and is a recurrent feature in plasticity models where the elasticity set has no bounded directions.\medskip

A direct consequence of this result is that, by means of the quasi-continuous representative of the stress with respect to the $H^1$-capacity, one can express the flow rule in a pointwise strong form. Indeed, since $\sigma(t) \in L^2(\O;\Ms)$ and $\dot p(t) \in \M(\overline \O;{\mathbb M}^{n{\times}n}_{\rm sym})$, the product between $\sigma(t)$ and $\dot p(t)$ is in general not well defined. This issue is usually overcome by introducing a distributional notion of duality $\sigma(t):\dot p(t)$ as in \cite{KT}. In the present situation the precise representative, denoted by $\tilde \sigma(t)$, is defined up to a set of zero capacity and thus it turns out to be $|\dot p(t)|$-measurable. This enables one to give a sense to the pointwise product of the (quasi-continuous) stress $\tilde \sigma(t)$ with the measure $\dot p(t)$, and, in particular, to express Hill's principle of maximal plastic work \eqref{eq:Hill} in a strong sense.\medskip

The paper is organised as follows. In Section~\ref{sec:2} we introduce the precise mathematical setting to formulate accurately the model of perfect plasticity, and state our main regularity result, Theorem~\ref{thm:main}. In Section~\ref{sec:3} we approximate the perfect plasticity model by means of a visco-plastic regularization, and establish the convergence of the solutions, as well as some (non uniform) regularity properties of the approximating solutions. Section~\ref{sec:4} is devoted to establish a chain rule type formula for the stress which, as explained above, is instrumental for the subsequent analysis. In Section~\ref{sec:5} we prove an estimate on the (visco-plastic) stress in $L^\infty(0,T;H^1_{\rm loc}(\O;\Ms))$, which is uniform with respect to the viscosity parameter. Owing to this estimate we complete the proof of Theorem~\ref{thm:main}. The last section is devoted to show the validity of the flow rule by means of the quasi-continuous representative of the stress.

\section{Mathematical formulation of the problem}\label{sec:2}

\subsection{Notation}

\subsubsection{Vectors and matrices}

If $a,b \in \R^n$, we write $a \cdot b$ for the Euclidean scalar product, and we denote by $|a|=\sqrt{a \cdot a}$ the associated norm.  
\medskip

We write $\mathbb M^{n \times n}$ for the set of real $n \times n$ matrices, and $\mathbb M^{n \times n}_{\rm sym}$ for that of all real symmetric $n \times n$ matrices.  Given two matrices $A$ and $B \in \mathbb M^{n \times n}$, we use the Frobenius scalar product $A:B=\sqrt{{\rm tr}(A^T B)}$ (where $A^T$ is the transpose of $A$, and ${\rm tr }A$ is its trace), and we denote by $|A|=\sqrt{A:A}$ the associated norm. If $A \in \mathbb M^{n \times n}$, we denote by $A_D:=A-\frac{1}{n}({\rm tr} A) {\rm Id}$ the deviatoric part of $A$, which is a trace free matrix. We recall that for any two vectors $a,b \in \R^n$, $a \otimes b :=a b^T\in \mathbb M^{n \times n}$ stands for the tensor product, and $a \odot b:= (a \otimes b + b \otimes a) /2\in \mathbb M^{n \times n}_{\rm sym}$ denotes  the symmetric tensor product.
 
\subsubsection{Functional spaces}

We use standard notation for Lebesgue and Sobolev spaces. 

\medskip

Let $X \subset \R^n$ be a locally compact set and $Y$ be an Euclidean space. We write $\mathcal{M}(X;Y)$ (or simply $\mathcal M(X)$ if $Y=\R$) for the space of bounded Radon measures in $X$ with values in $Y$, endowed with the  norm $|\mu|(X)$, where $|\mu|\in \mathcal{M}(X)$ is the variation of the measure $\mu$. The Lebesgue measure in $\R^n$ is denoted by $\LL^n$, and the $(n-1)$-dimensional Hausdorff measure by $\HH^{n-1}$.  

\medskip

If $U \subset \R^n$ is an open set, $BD(U)$ stands for the space of functions of bounded deformation in $U$, {\it i.e.}, $u\in BD(U)$ if $u \in L^1(U;\R^n)$ and $Eu\in \M(U;\mathbb M^{n \times n}_{\rm sym})$, where $Eu:=(Du+Du^T)/2$ and $Du$ is the distributional derivative of $u$. 
We recall that, if $U \subset \R^n$ is bounded and has a Lipschitz boundary, $BD(U)$ can be embedded into $L^{n/(n-1)}(U;\R^n)$. 
We refer to \cite{T} for general properties of this space.

\subsubsection{Capacity}

We finally recall the definition and several facts about capacity (see \cite{AH}). Let $\O \subset \R^n$ be an open set. The capacity of a set $A \subset \O$ in $\O$ is defined by
$$\Cap(A):=\inf\left\{\int_\O |\nabla u|^2\, dx : u \in H^1_0(\O), \; u \geq 1 \text{ $\LL^n$-a.e.\ in a neighborhood of }A\right\}.$$
One of the interests of capacity is that it enables one to give an accurate sense to the pointwise value of Sobolev functions (see \cite[Section~6.1]{AH}). More precisely, every $u \in H^1(\O)$ has a quasi-continuous representative $\tilde u$, which is uniquely defined except on a set of capacity zero in $\O$. It means that $\tilde u=u$ $\LL^n$-a.e.\ in $\O$, and that, for each $\e>0$, there exists a closed set $A_\e \subset \O$ such that $\Cap(\O \setminus A_\e)<\e$ and $\tilde u|_{A_\e}$ is continuous on $A_\e$. In addition (see \cite[Theorem~6.2.1]{AH}), there exists a  Borel set $Z \subset \O$ with $\Cap(Z)=0$ such that
$$
\lim_{r\to0^+}\frac{1}{\LL^n(B_r(x) \cap \O)} \int_{B_r(x) \cap \O} u(y)\, dy =\tilde u(x) \quad \text{ for all } x\in \O \setminus Z.
$$

\subsection{Description of the model}

\subsubsection{The reference configuration}  We denote by
\begin{equation}\label{eq:Omega}\tag{$A_1$}
\O \subset \R^2 \text{ a bounded connected open set with Lipschitz boundary}
\end{equation}
the reference configuration of an elasto-plastic material.

\subsubsection{Boundary condition}

We assume that the body is subjected to a time-dependent boundary displacement, which is the trace on $\partial \O$ of a function $w(t):\O \to \R^2$ with 
\begin{equation}\label{eq:w}\tag{$A_2$}
w \in AC([0,T];H^1(\O;\R^2)).
\end{equation} 

\subsubsection{The elastic energy}

We consider an isotropic body whose fourth order elasticity tensor $\CC$ is given by
\begin{equation}\label{eq:isotropic}\tag{$A_3$}
\CC e= \lambda ({\rm tr} e) {\rm Id} +2\mu e \quad \text{ for all }e \in \Ms,
\end{equation}
where $\lambda$ and $\mu$ are the Lam\'e coefficients satisfying $\mu>0$ and $\lambda + \mu>0$. Note that there exist two constants $c_0$, $c_1>0$ such that
\begin{equation}\label{eq:coerc-Q}
c_0 |e|^2\leq  \CC e: e \leq c_1|e|^2 \quad \text{ for all }e \in \Ms.
\end{equation}
Setting $K_0:=\lambda +\mu$, the inverse of $\CC$ can be represented as
$$\CC^{-1}\sigma=\frac{1}{4K_0}({\rm tr} \sigma) {\rm Id}+ \frac{1}{2\mu}\sigma_D\quad \text{ for all }\sigma \in \Ms.$$

We define the elastic energy, for all $e \in L^2(\O;\Ms)$, by
$$
\Q(e):= \frac12 \int_\O \CC e(x):e(x)\, dx.
$$

\subsubsection{The elasticity set}

In this paper we are interested in the  {\it Drucker-Prager} and {\it Mohr-Coulomb}\footnote{Note that in dimension $n=2$ the two models are equivalent because of the algebraic identity $\sqrt{2} |\sigma_D|=\sigma_{\rm max}-\sigma_{\rm min}$, where $\sigma_{\rm max}$ (resp.\ $\sigma_{\rm min}$) is the largest (resp.\ lowest) eigenvalue of $\sigma_D$.}
 models, where the elasticity domain is a closed and convex cone with vertex lying on the axis of hydrostatic stresses given by
\begin{equation}\label{eq:K}\tag{$A_4$}
K:=\{\sigma \in\Ms :  |\sigma_D|+ \alpha {\rm tr} \sigma \leq \kappa\}.
\end{equation}
In the previous formula $\alpha>0$ and $\kappa>0$ are positive constants related to the cohesion and the coefficient of internal friction of the material, respectively. 

\subsubsection{External forces}

We consider a time-dependent body load $f(t):\O \to \R^2$ satisfying 
\begin{equation}\label{eq:f}\tag{$A_5$}
f \in AC([0,T];L^2(\O;\R^2)),
\end{equation}
which satisfies the usual safe-load condition: there exist $\chi\in AC([0,T];L^2(\O;\Ms))$ and a constant $\delta \in (0,\kappa)$ such that for every $t\in[0,T]$
\begin{equation}\label{eq:safe}\tag{$A_6$}
\begin{cases}
-\div\chi(t) = f(t) \quad \text{ in }\O,\\
|\chi_D(t)|+ \alpha {\rm tr} \chi(t) \leq \kappa-\delta\quad \text{ in }\O.
\end{cases}
\end{equation}

\subsubsection{The dissipation energy}

We define the support function $H:\Ms\to [0,+\infty]$ of $K$ by
$$H(p):=\sup_{\sigma \in K} \sigma {\,:\,} p \quad \text{ for all } p \in \Ms.$$
Since $K$ is closed and convex, $H$ is convex, lower semicontinuous, and positively $1$-homogeneous. In addition, since $0$ belongs to the interior of $K$, the functions $H$ enjoys the following coercivity property: there exists $c_2>0$
$$H(p) \geq  c_2 |p|\quad \text{ for all }p \in\Ms .$$

It is easy to establish the following explicit formula for the function $H$.
\begin{lem}\label{lem:H}
For all $p \in \Ms$,
$$H(p)=
\begin{cases}
\ds\frac{\kappa\, {\rm tr }\, p}{2\alpha} & \ds \text{ if } |p_D| \leq \frac{{\rm tr }\, p}{2\alpha},\\[0.3cm]
+\infty & \ds \text{ if } |p_D| > \frac{{\rm tr }\, p}{2\alpha}.
\end{cases}
$$
\end{lem}

The dissipated energy functional is then defined, for all $p \in L^1(\O;\Ms)$, by 
$$
\HH(p):= \int_\O H(p(x))\, dx.
$$
As a consequence of the previous properties of $H$, we infer that $\HH$ is sequentially weakly lower semicontinuous in $L^1(\O;\Ms)$. Since $L^1(\O;\Ms)$ is not reflexive (bounded sequences in that space are only weakly* sequentially compact in the space of measures), it will also be useful to extend the definition of $\HH$ when $p \in \M(\overline \O;\Ms)$. According to \cite{GS}, we define the non-negative Borel measure
$$
H(p):= H\left(\frac{dp}{d|p|}\right)|p|,
$$
where $\frac{dp}{d|p|}$ is the Radon-Nikodym derivative of $p$ with respect to its variation $|p|$.
In general, the measure $H(p)$ is not even locally finite. However, if further $H(p)$ has finite mass, {\it i.e.}, if $H(p)$ is a bounded Radon measure, we can define the dissipation functional
$$
\HH(p):=H(p)(\overline \O).
$$
In that case, the results of \cite{DT1,DT2} apply  and $H(p)$ can be expressed by means of a duality formula. 
If $H(p) \in \M(\overline \O)$, we get that (see \cite{BM})
\begin{equation}\label{eq:duality-form-phi}
\int_{\overline \O} \varphi \, d H(p)=\sup\left\{\int_{\overline \O}\varphi\sigma{\,:\,}dp \ : \sigma \in \C^\infty(\overline \O;K) \right\},
\end{equation}
for any $\varphi \in \C(\overline \O)$ with $\varphi \geq 0$, and in particular
\begin{equation}\label{eq:duality-form}
 \HH(p)=\sup\left\{\int_{\overline \O}\sigma{\,:\,}dp \ : \sigma \in \C^\infty(\overline \O;K) \right\}.
 \end{equation}
Note also that the Reshetnyak Theorem (see \cite[Theorem~2.38]{AFP}) applies here, so that $\HH$ is sequentially weakly* lower semicontinuous in $\M(\overline \O;\Ms)$. 

\subsubsection{Spaces of admissible fields}

Given a prescribed boundary displacement $\hat w \in H^1(\O;\R^2)$, we will consider the following spaces of kinematically admissible fields:
\begin{multline*}
\AR(\hat w):=\big\{(v,\eta,q) \in H^1(\O;\R^2) \times L^2(\O;\Ms) \times L^2(\O;\Ms):\\
Ev=\eta+q \text{ a.e.\ in }\O,\quad  v=\hat w\text{ $\HH^1$-a.e.\ on }\partial \O \big\},
\end{multline*}
and\begin{multline*}
\A(\hat w):=\Big\{(v,\eta,q) \in BD(\O) \times L^2(\O;\Ms) \times \M(\overline \O;\Ms) :\\
Ev=\eta+q \text{ in }\O, \quad q=(\hat w-v) \odot \nu\;  \HH^1 \text{ on }\partial \O\Big\},
\end{multline*}
where $\nu$ is the outer unit normal to $\partial\O$. 

The space of plastically admissible stresses is defined by
$$\K:=\{\tau \in L^2(\O;\Ms) : \tau(x) \in K \text{ for a.e.\ }x \in \O\},$$
and the space of statically admissible stresses is given by
$$\mathcal S :=\{ \tau \in L^2(\O;\Ms) : \div\tau\in L^2(\O;\R^2)\}.$$

\subsubsection{Stress/strain duality}

The duality pairing between stresses and plastic strains is {\em a priori} not well defined, since the former are only squared Lebesgue integrable, while the latter are possibly singular measures. Following \cite{KT}, we define the following distributional notion of duality. 

\begin{defn}\label{def:duality}
Let $\sigma\in \mathcal S$ and $(u,e,p) \in \A(\hat w)$ with $\hat w \in H^1(\O;\R^2)$.
We define the distribution $[\sigma{\,:\,}p] \in \mathcal D'(\R^2)$ supported in $\overline \O$ by
\begin{equation}\label{eq:duality}
\langle [\sigma{\,:\,}p], \varphi\rangle = \int_\O \varphi(\hat w-u)\cdot\div\sigma\, dx +\int_\O \sigma:[(\hat w-u)\odot\nabla\varphi]\, dx
+\int_\O \sigma:(E\hat w-e)\varphi\, dx
\end{equation}
for every $\varphi\in \C^\infty_c(\R^2)$. 
The duality product is then defined as
$$\langle\sigma,p\rangle:=\langle[\sigma{\,:\,}p], 1\rangle=\int_\O (\hat w-u)\cdot\div\sigma\, dx+\int_\O \sigma:(E\hat w-e)\, dx.$$
\end{defn}

\begin{rem}\label{product}
Note that the first and second integrals in \eqref{eq:duality} are well defined since
$BD(\O)$ is embedded into $L^2(\O;\R^2)$ for $n=2$. Moreover,
according to the integration by parts formula in $BD(\O)$ (see \cite[Theorem~3.2]{B}), if $\sigma \in \mathcal S\cap\C^1(\overline \O;\Ms)$, we have
\begin{equation}\label{eq:duality-sigma-smooth}
\langle [\sigma{\,:\,}p], \varphi\rangle=\int_{\overline \O}  \varphi\sigma :dp \quad \text{ for all }\varphi \in \C^\infty_c(\R^2).
\end{equation}
A convolution argument shows that \eqref{eq:duality-sigma-smooth} remains true provided $\sigma \in\mathcal S\cap \C(\overline \O;\Ms)$ and $\varphi \in \C^\infty(\overline \O)$.
\end{rem}

Using this notion of stress/strain duality, the duality formulas  \eqref{eq:duality-form-phi} and \eqref{eq:duality-form} can be now extended to less regular statically and plastically admissible stresses. 
If $p\in\M(\overline\O;\Ms)$ with $\HH(p)<+\infty$,
$$\int_{\overline \O}\varphi\, dH(p)=\sup\Big\{\langle [\sigma{\,:\,}p],\varphi \rangle  : \sigma \in \K\cap \mathcal S\Big\},$$
for all $\varphi \in \C^\infty(\overline \O)$ with $\varphi \geq 0$, and in particular,
$$\HH(p)=\sup\Big\{\langle \sigma, p\rangle  : \sigma \in \K \cap \mathcal S\Big\}.$$

The following result establishes a coercivity property of the functional $p \mapsto \HH(p) -\langle \chi(t),p\rangle$ (see {\it e.g.} \cite[Proposition~6.1]{BM}).

\begin{prop}\label{prop:QS1}
Let $\hat w\in H^1(\O;\R^2)$ and  $(u,e,p)\in \A(\hat w)$. Then there exists a constant $C_{\delta,\alpha}$, depending on $\delta$ and $\alpha$, such that the following coercivity estimate holds:
$$
\HH(p) -\langle\chi(t),p\rangle \geq C_{\delta,\alpha} \|p\|_{\M(\overline \O;\Ms)}$$
for every $t\in[0,T]$.
\end{prop}

\subsubsection{Initial condition}

We finally consider an initial datum $(u_0,e_0,p_0) \in  \A(w(0))$ and $\sigma_0:=\CC e_0$ satisfying the stability conditions
\begin{equation}\label{eq:init2}\tag{$A_7$}
\sigma_0 \in \K, \quad -\div \sigma_0=f(0) \text{ in } \O.
\end{equation}

\bigskip

We are now in position to state the existence result obtained in \cite{BM}.

\begin{thm}\label{thm:nocap2}
Assume  \eqref{eq:Omega}--\,\eqref{eq:init2}. 
Then there exist 
$$
\begin{cases}
u \in AC([0,T];BD(\O)),\\
e, \; \sigma \in AC([0,T];L^2(\O;\Ms)),\\
p \in AC([0,T]; \M(\overline\O;\Ms)),
\end{cases}
$$
with $$(u(0),e(0),p(0))=(u_0,e_0,p_0)$$
that satisfy, for all $t \in [0,T]$:
\begin{itemize}
\item[(i)] the kinematic compatibility 
$$
\begin{cases}
Eu(t)=e(t)+p(t) \text{ in }\O,\\
p(t)=(w(t)-u(t))\odot \nu \HH^1 \text{ on }\partial\O,
\end{cases}
$$
\item[(ii)] the static and plastic admissibility
$$
\begin{cases}
\sigma(t)=\CC e(t),\\
- \div\sigma(t)=f(t) \text{  in }\O,\\
 \sigma(t) \in \K,
 \end{cases}
$$
\item[(iii)] the energy balance
\begin{equation}\label{eq:ee}
\Q(e(t)) + \int_0^t \HH (\dot p(s))\, ds = \Q(e_0) 
  + \int_0^t\int_\O \sigma:E\dot w\, dx \, ds +\int_0^t \int_\O f\cdot (\dot u-\dot w)\, dx\, ds.
\end{equation}
\end{itemize}
Moreover, the stress $\sigma$ is unique, and for a.e.\ $t \in [0,T]$ the distribution $[\sigma(t):\dot p(t)]$ is well defined, and it is a measure in $\M(\overline \O)$ satisfying Hill's principle of maximum plastic work
\begin{equation}\label{eq:fr}
H(\dot p(t))= [\sigma(t):\dot p(t)] \quad \text{ in }\M(\overline \O).
\end{equation}
\end{thm}

\begin{rem}
According to Lemma~\ref{lem:H}, the flow rule can be equivalently written as
$$\frac{\kappa\, {\rm tr }\, \dot p(t)}{2\alpha} =[\sigma(t):\dot p(t)], \quad  \text{ and }\quad  |\dot p_D(t)| \leq \frac{{\rm tr }\, \dot p(t)}{2\alpha} \quad \text{ in }\M(\overline \O).$$
\end{rem}

The main result of this work is the following regularity result.
\begin{thm}\label{thm:main}
Assume \eqref{eq:Omega}--\,\eqref{eq:init2} and that $\alpha \neq 1/\sqrt 2$ in \eqref{eq:K}. Under the additional hypotheses that $w \in H^1([0,T];H^2(\O;\R^2))$, $\chi \in W^{1,\infty}([0,T];L^\infty(\O;\Ms)) \cap H^1([0,T];H^1(\O;\Ms))$, $f \in L^\infty(0,T;H^1(\O;\R^2))\cap L^2(0,T;H^2(\O;\R^2)) \cap L^\infty(\O \times (0,T);\R^2)$ and  $e_0 \in H^1_{\rm loc}(\O;\mathbb M^{2 \times 2}_{\rm sym})$, the stress tensor satisfies
$$\sigma\in L^\infty(0,T;H^1_{\rm loc}(\O;\mathbb M^{2 \times 2}_{\rm sym})).$$
\end{thm}

\section{Perzyna visco-plastic approximations}\label{sec:3}

In order to prove Theorem~\ref{thm:main}, we will need to consider a regularized problem. This will be done by means of a so-called Perzyna visco-plastic approximation. The following result, formulated here in a modern language, has been established in \cite{Suquet}. 

\medskip

Since the initial data $(u_0,e_0,p_0)$ given in \eqref{eq:init2} does not belong to the right energy space associated to the visco-plastic model, we first need to regularize it. According to \cite[Lemma~5.1]{DMDSS}, there exists a sequence $(u_{0,\e}) \subset H^1(\O;\R^2)$ such that $u_{0,\e}=w(0)$ $\HH^1$-a.e.\ on $\partial \O$, $u_{0,\e}\to u_0$ strongly in $L^1(\O;\R^2)$, and $Eu_{0,\e} \wto Eu_0$ weakly* in $\M(\overline \O;\Ms)$. Setting $p_{0,\e}=Eu_{0,\e} -e_0$, we get that $(u_{0,\e},e_0,p_{0,\e}) \in \AR(w(0))$. 

\begin{prop}\label{prop:perzyna}
Assume \eqref{eq:Omega}--\,\eqref{eq:init2}. Let $\e>0$ and let $(u_{0,\e},e_0,p_{0,\e}) \in \AR(w(0))$ be constructed as above. Then there exists a unique triple 
$$(u_\e,e_\e,p_\e) \in AC([0,T];H^1(\O;\R^2)) \times AC([0,T];L^2(\O;\Ms)) \times  AC([0,T];L^2(\O;\Ms))$$ 
such that $(u_\e(0),e_\e(0),p_\e(0))=(u_{0,\e},e_0,p_{0,\e})$,
for all  $t \in [0,T]$
$$(u_\e(t),e_\e(t),p_\e(t)) \in \AR(w(t)), \quad \sigma_\e(t)=\CC e_\e(t),\quad -\div \sigma_\e(t)=f(t) \quad \text{ in }\O,$$
and for a.e.\ $t \in [0,T]$
\begin{equation}\label{eq:reg-flow-rule}
\sigma_\e(t) - \e\dot p_\e(t) \in \partial H(\dot p_\e(t)) \quad \text{ in }\O.
\end{equation}
\end{prop}

\begin{rem}\label{rem:convex-anal}
For every $q \in \Ms$, we define the function
$$H_\e(q):=H(q)+\frac{\e}{2}|q|^2.$$
The convex conjugate of $H_\e$ is given, for all $\tau \in \Ms$, by 
$$H^*_\e(\tau)=\frac{|\tau-P_K(\tau)|^2}{2\e},$$
where $P_K$ stands for the orthogonal projection onto the nonempty closed convex set $K$. The function $H^*_\e$ turns out to be of class $\C^1$ and its differential is given by
\begin{equation}\label{eq:psi*}
DH_\e^*(\tau)=\frac{\tau-P_K(\tau)}{\e}.
\end{equation}

With these notation, the flow rule \eqref{eq:reg-flow-rule}, can be equivalently written, for a.e.\ $t \in [0,T]$, as 
$$\sigma_\e(t) \in \partial H_\e(\dot p_\e(t)) \quad \text{ in } \O,$$
or still, by convex analysis,
\begin{equation}\label{eq:reg-flow-rule3}
\dot p_\e(t) =DH^*_\e(\sigma_\e(t)) \quad \text{ in } \O.
\end{equation}
\end{rem}

\medskip 

We will show that the solution $(u_\e,e_\e,p_\e)$ of the visco-plastic model given by Proposition~\ref{prop:perzyna} converges to a solution of the perfectly plastic model, in the sense of Theorem~\ref{thm:nocap2}. 

\begin{prop}\label{prop:convergence}
Assume that \eqref{eq:Omega}--\,\eqref{eq:init2} hold, and in addition that $w \in H^1([0,T];H^1(\O;\R^2))$ and $\chi\in W^{1,\infty}([0,T];L^\infty(\O;\Ms))$. Then, up to a subsequence (not relabeled), $(u_\e,e_\e,p_\e) \wto (u,e,p)$ weakly* in $H^1([0,T];BD(\O))\times H^1([0,T];L^2(\O;\Ms)) \times H^1([0,T];\M(\overline \O;\Ms))$, where $(u,e,p)$ is a solution of the perfectly plastic model as in Theorem~\ref{thm:nocap2}.
\end{prop}

Note that this result was already proven in \cite{Suquet} for different type of elasticity sets that are bounded in the direction of deviatoric stresses (see also \cite{DMDSM}). However we give below a slightly different and simplified argument since some finer estimates established along the proof will be useful in that of our regularity result Theorem~\ref{thm:main}.

\subsection{{\it A priori} estimates}

We first establish  some uniform {\it a priori} estimates which will enable one to get weak compactness on the families $(u_\e)_{\e>0}$, $(e_\e)_{\e>0}$, and $(p_\e)_{\e>0}$.

\subsubsection{First energy estimates}

Standard arguments show that the following energy balance holds: for all $t \in [0,T]$,
\begin{multline}\label{eq:energy-balance0}
\Q(e_\e(t)) + \int_0^t \HH(\dot p_\e(s))\, ds + \e \int_0^t \int_\O |\dot p_\e|^2\, dx \, ds\\
 = \Q(e_0) + \int_0^t \int_\O \sigma_\e:E\dot w\, dx \, ds  + \int_0^t \int_\O f\cdot (\dot u_\e-\dot w)\, dx \, ds,
 \end{multline}
or still, using the safe load condition \eqref{eq:safe}, together with an integration by parts in time,
\begin{multline}\label{eq:energy-balance}
\Q(e_\e(t)) + \int_0^t \HH(\dot p_\e(s))\, ds -\int_0^t \int_\O \chi:\dot p_\e\, dx \, ds+ \e \int_0^t \int_\O |\dot p_\e|^2\, dx \, ds\\
 = \Q(e_0) + \int_0^t \int_\O \sigma_\e:E\dot w\, dx \, ds  - \int_0^t \int_\O\dot \chi:( e_\e - E w)\, dx \, ds\\
 + \int_\O \chi(t):( e_\e(t) - E w(t))\, dx- \int_\O \chi(0):( e_0 - E w(0))\, dx.
\end{multline}
Therefore, an application of Proposition~\ref{prop:QS1} leads to the following first energy estimates:
\begin{equation}\label{eq:first-estim}
\sup_{\e>0}\left(\|e_\e\|_{L^\infty(0,T;L^2(\O;\Ms))}+ \|\dot p_\e\|_{L^1(0,T;L^1(\O;\Ms))} \right.
\left.+ \sqrt\e\|\dot p_\e\|_{L^2(0,T;L^2(\O;\Ms))}\right)<+\infty.
\end{equation}

\subsubsection{Second energy estimates}

Writing the additive decomposition for the rates yields $E\dot u_\e=\dot e_\e + \dot p_\e$ in $\O \times (0,T)$,
taking the scalar product with $\dot \sigma_\e-\dot \chi$ and integrating over $\O \times (0,T)$ leads to
\begin{multline}\label{MGadd}
2\int_0^t \Q(\dot e_\e(s)-\CC^{-1}\dot \chi(s))\, ds\\
 = \int_0^t \int_\O (\dot\sigma_\e-\dot \chi) : E\dot u_\e\, dx \, ds - \int_0^t \int_\O(\dot \sigma_\e-\dot\chi) : \dot p_\e\, dx \, ds-\int_0^t\int_\O \CC^{-1}\dot \chi:(\dot \sigma_\e - \dot \chi)\, dx \, ds.
 \end{multline}
By integrating by parts in space and by using the equilibrium equation, together with the safe load condition \eqref{eq:safe}, 
we can rewrite the first integral at the right hand side of \eqref{MGadd} as 
$$\int_0^t \int_\O(\dot \sigma_\e-\dot\chi) : E\dot u_\e\, dx \, ds=\int_0^t \int_\O(\dot \sigma_\e-\dot\chi) : E\dot w\, dx \, ds.$$
Using the flow rule \eqref{eq:reg-flow-rule3}, we have $\dot p_\e=DH_\e^*(\sigma_\e)$, so that, owing to the chain rule formula and a derivation under the integral sign, we get for a.e.\ $t \in [0,T]$,
$$\int_\O \dot \sigma_\e(t):\dot p_\e(t)\, dx=\int_\O DH_\e^*(\sigma_\e):\dot \sigma_\e(t)\, dx=\frac{d}{dt}\int_\O H_\e^*(\sigma_\e(t))\, dx.$$
Thus, since $\sigma_0 \in \K$,
$$-\int_0^t\int_\O \dot \sigma_\e(t):\dot p_\e(t)\, dx=-\int_\O H_\e^*(\sigma_\e(t))\, dx+\int_\O H_\e^*(\sigma_0)\, dx \leq \int_\O H_\e^*(\sigma_0)\, dx=0.$$
Finally, using the safe load condition \eqref{eq:safe}, together with \eqref{eq:coerc-Q} and  H\"older's inequality, we obtain
from \eqref{MGadd} the following estimate:
\begin{multline*}
\|\dot \sigma_\e-\dot \chi\|^2_{L^2(0,T;L^2(\O;\Ms))}  \leq \|\dot \chi\|_{L^\infty(\O \times (0,T);\Ms)} \|\dot p_\e\|_{L^1(0,T;L^1(\O;\Ms))}\\
+C\left(\|E\dot w\|_{L^2(0,T;L^2(\O;\Ms))} + \|\dot \chi\|_{L^2(0,T;L^2(\O;\Ms))}\right)\|\dot \sigma_\e-\dot \chi\|_{L^2(0,T;L^2(\O;\Ms))},
\end{multline*}
where $C>0$ is independent of $\e$. By the first energy estimate \eqref{eq:first-estim} we infer that
\begin{equation}\label{eq:second-estim}
\sup_{\e>0}\|\dot \sigma_\e\|_{L^2(0,T;L^2(\O;\Ms))} <+\infty.
\end{equation}

\subsubsection{Third energy estimates}

Writing next the energy balance \eqref{eq:energy-balance} between two arbitrary times $t_1$ and $t_2$, and integrating by parts in times  yield
\begin{multline*}
\Q(e_\e(t_2)) + \int_{t_1}^{t_2} \HH(\dot p_\e(s))\, ds -\int_{t_1}^{t_2} \int_\O \chi:\dot p_\e\, dx \, ds+ \e \int_{t_1}^{t_2} \int_\O |\dot p_\e|^2\, dx \, ds\\
 = \Q(e_\e(t_1)) + \int_{t_1}^{t_2}  \int_\O \sigma_\e:E\dot w\, dx \, ds  + \int_{t_1}^{t_2}  \int_\O \chi:( \dot e_\e - E \dot w)\, dx \, ds.
 \end{multline*}
By applying Proposition~\ref{prop:QS1}, together with estimate \eqref {eq:first-estim}, 
we deduce that
\begin{multline*}
C_{\delta,\alpha}\|p_\e(t_2)-p_\e(t_1)\|_{L^1(\O;\Ms)}\leq C_{\delta,\alpha}\int_{t_1}^{t_2} \|\dot p_\e(s)\|_{L^1(\O;\Ms)}\, ds \\
\leq  \int_{t_1}^{t_2} \HH(\dot p_\e(s))\, ds -\int_{t_1}^{t_2} \int_\O \chi:\dot p_\e\, dx \, ds
 \leq\frac12 \int_\O (\sigma_\e(t_1)+\sigma_\e(t_2)): (e_\e(t_1)-e_\e(t_2)) \, dx\\
 + \int_{t_1}^{t_2}  \int_\O \sigma_\e:E\dot w\, dx \, ds  + \int_{t_1}^{t_2}  \int_\O \chi:( \dot e_\e - E \dot w)\, dx \, ds\leq   \int_{t_1}^{t_2}h_\e(s)\, ds,
\end{multline*}
where $h_\e \in L^2(0,T)$ is defined, for a.e.\ $s \in [0,T]$, by
$$h_\e(s):
=M \left(\|\dot e_\e(s)\|_{L^2(\O;\Ms)} +\|E\dot w(s)\|_{L^2(\O;\Ms)} \right)-\int_\O \chi(s):E\dot w(s)\, dx,$$
and
$$
M:=\sup_{\e>0}\|\sigma_\e\|_{L^\infty(0,T;L^2(\O;\Ms))} +\|\chi\|_{L^\infty(0,T;L^2(\O;\Ms))}.
$$
According to \cite[Proposition~A.3]{Br} and  \eqref{eq:second-estim}, this implies that $p_\e \in H^1([0,T];L^1(\O;\Ms))$,
\begin{equation}\label{eq:third-estim}
\sup_{\e>0}\|\dot p_\e\|_{L^2(0,T;L^1(\O;\Ms))} \leq \sup_{\e>0}\|h_\e\|_{L^2(0,T)} <+\infty,
\end{equation}
and owing to the Poincar\'e-Korn inequality (see \cite[Chap.\ 2, Rmk.\ 2.5(ii)]{T}), that 
\begin{equation}\label{eq:third-estim-bis}
\sup_{\e>0}\|\dot u_\e\|_{L^2(0,T;BD(\O))}<+\infty.
\end{equation}

\subsection{Passage to the limit in $\e$} 

According to \eqref{eq:second-estim}, \eqref{eq:third-estim}, and \eqref{eq:third-estim-bis}, there exists a subsequence (not relabeled) and functions $u \in H^1([0,T];BD(\O))$, $e \in H^1([0,T];L^2(\O;\Ms))$, and $p \in H^1([0,T];\M(\overline \O;\Ms))$ such that
$$\begin{cases}
u_\e \wto u & \text{ weakly* in }H^1([0,T];BD(\O)),\\
e_\e \wto e & \text{ weakly in }H^1([0,T];L^2(\O;\Ms)),\\
p_\e \wto p & \text{ weakly* in }H^1([0,T];\M(\overline \O;\Ms)).
\end{cases}$$
An application of the Ascoli-Arzel\`a Theorem also shows that
$$\begin{cases}
u_\e(t) \wto u(t) & \text{ weakly* in }BD(\O),\\
e_\e(t) \wto e(t) & \text{ weakly in }L^2(\O;\Ms),\\
p_\e(t) \wto p(t) & \text{ weakly* in }\M(\overline \O;\Ms)
\end{cases}$$
for all $t \in [0,T]$.
Since $(u_\e(0),e_\e(0),p_\e(0))=(u_{0,\e},e_0,p_{0,\e})$, by passing to the limit we deduce that the initial condition $(u(0),e(0),p(0))=(u_{0},e_0,p_{0})$ is satisfied. 
Note also that, according to \cite[Lemma~2.1]{DMDSM}, we have that $(u(t),e(t),p(t)) \in \A(w(t))$ for all $t \in [0,T]$. Moreover, defining $\sigma:=\CC e$, we get that $-\div \sigma(t)=f(t)$ in $\O$ for all $t \in [0,T]$. 

\medskip

We then prove the validity of the stress constraint. Owing to \eqref{eq:first-estim} and \eqref{eq:second-estim}, the fact that $P_K$ is $1$-Lipschitz and that $0 \in K$, up to another subsequence, we have
$$\begin{cases}
P_K(\sigma_\e) \wto \tau & \text{ weakly in } H^1([0,T];L^2(\O;\Ms)),\\
P_K(\sigma_\e(t)) \wto \tau(t) & \text{ weakly in $L^2(\O;\Ms)$ for all $t \in [0,T]$},
\end{cases}$$
for some $\tau \in H^1([0,T];L^2(\O;\Ms))$ with $\tau(t) \in \K$ for all $t \in [0,T]$. According to the flow rule \eqref{eq:reg-flow-rule3}, \eqref{eq:psi*}, and \eqref{eq:first-estim}, we get
$$\sigma_\e-P_K(\sigma_\e)=\e\dot p_\e \to 0 \text{ strongly in }L^2(0,T;L^2(\O;\Ms)),$$
which implies that $\sigma=\tau$ and thus, that $\sigma(t) \in \K$ for all $t \in [0,T]$.

\medskip

It remains to show the energy balance. According to the weak convergences established so far, it is possible to pass to the lower limit in the energy balance \eqref{eq:energy-balance0} to get, for all $t \in [0,T]$,
$$\Q(e(t)) + \int_0^t \HH (\dot p(s))\, ds \leq \Q(e_0) 
  + \int_0^t\int_\O \sigma:E\dot w\, dx \, ds +\int_0^t \int_\O f\cdot (\dot u-\dot w)\, dx\, ds.$$
  
 The converse inequality, and thus equality \eqref{eq:ee}, can be proved using a standard argument of rate independent processes (see \cite[Theorem~4.7]{DMDSM}) by noticing that the conditions 
 $$\sigma(t) \in \K \cap \mathcal S,\quad -\div \sigma(t)=f(t) \text{ in }\O$$
 are equivalent to the minimality property
 $$\Q(e(t)) -\int_\O f(t) \cdot u(t)\, dx \leq \Q(\eta) + \HH(q-p(t))-\int_\O f(t)\cdot v\, dx$$
 for all $(v,\eta,q) \in \A(w(t))$. This can be seen by adapting the argument of \cite[Theorem~3.6]{DMDSM}, which only requires the stress-strain duality pairing to be well defined.

\subsection{Higher regularity properties}

In the proof of our regularity result Theorem~\ref{thm:main}, we will need some higher regularity properties for the solution of the visco-plastic model.
\begin{prop}\label{prop:higher-reg}
Assume that \eqref{eq:Omega}--\eqref{eq:init2} hold, and, in addition, that $w \in H^1([0,T];H^2(\O;\R^2))$, $\chi\in W^{1,\infty}([0,T];L^\infty(\O;\Ms)) \cap H^1([0,T];H^1(\O;\Ms))$, $f \in  L^\infty(0,T;H^1(\O;\R^2))$, and $e_0 \in H^1_{\rm loc}(\O;\Ms)$. Then 
$$\dot u_\e \in L^2(0,T;H^2_{\rm loc}(\O;\R^2)), \quad e_\e,\; \sigma_\e \in H^1([0,T];H^1_{\rm loc}(\O;\Ms)), \quad \dot p_\e\in L^2(0,T;H^1_{\rm loc}(\O;\Ms)).$$
\end{prop}

Let us introduce the following notation: given a generic function $\phi:\R^2 \to \R$, we write
$$\partial_k^h \phi(x):=\frac{\phi(x+he_k)-\phi(x)}{h} \quad \text{for $k \in \{1,2\}$ and $h>0$ small.}$$

\subsubsection{Fourth energy estimates}

Using the linearity of the equilibrium equations, we get that for all $t \in [0,T]$,
$$-\div (\partial_k^h \sigma_\e(t)) =\partial_k^h f(t)\quad \text{ a.e.\ in }\{x \in \O : \dist(x,\partial \O) >h\}.$$
Let $\varphi \in \C^\infty_c(\O)$ be a cut-off function with $h < \dist({\rm supp \,} \varphi,\partial \O)$. We multiply the previous equation by $\varphi^2 \partial_k^h \dot u_\e$ and integrate by parts in the space variables. In this way we get
$$\int_0^t\int_\O \varphi^2 \partial_k^h \sigma_\e : \partial_k^h E \dot u_\e \, dx\, ds +\int_0^t \int_\O \partial_k^h \sigma_\e : (\nabla \varphi^2 \odot \partial_k^h \dot u_\e)\, dx\; ds=\int_0^t\int_\O \varphi^2 \partial_k^h f \cdot  \partial_k^h \dot u_\e\, dx\, ds.$$
Using the additive decomposition $\partial_k^h E \dot u_\e=\partial_k^h\dot e_\e + \partial_k^h\dot p_\e$, 
the regularized flow rule \eqref{eq:reg-flow-rule}, and the monotonicity of the subdifferential of a convex function, we obtain
\begin{multline*}
 \Q(\varphi  \partial_k^h e_\e(t))+\e \int_0^t \int_\O\varphi^2 |\partial_k^h \dot p_\e|^2\, dx \, ds\\
 \leq  \Q(\varphi  \partial_k^h e_0)+\int_0^t\int_\O \varphi^2 \partial_k^h f \cdot  \partial_k^h \dot u_\e\, dx\, ds -2\int_0^t \int_\O \varphi \partial_k^h \sigma_\e : (\nabla \varphi \odot \partial_k^h \dot u_\e)\, dx\; ds\\
 \leq C\left(\|e_0\|^2_{H^1_{\rm loc}(\O;\Ms)} + \|f\|_{L^\infty(0,T;H^1(\O;\R^2))} \|\nabla \dot u_\e\|_{L^1(0,T;L^2(\O;\Ms))}\right.\\
 \left. + \|\varphi \partial_k^h \sigma_\e\|_{L^\infty(0,T;L^2(\O;\Ms))} \|\nabla \dot u_\e\|_{L^1(0,T;L^2(\O;\Ms))} \right),
\end{multline*}
where, here and in the following, $C$ denotes a positive constant independent of $h$ and $\e$, but possibly depending on $\varphi$. Thanks to \eqref{eq:coerc-Q} and Young's inequality, we infer that
\begin{multline*}
 \|\varphi \partial_k^h \sigma_\e\|^2_{L^\infty(0,T;L^2(\O;\Ms))}+\e \|\varphi \partial_k^h \dot p_\e\|^2_{L^2(0,T;L^2(\O;\Ms))}\\
 \leq C\left(\|e_0\|^2_{H^1_{\rm loc}(\O;\Ms)} + \|f\|^2 _{L^\infty(0,T;H^1(\O;\R^2))}+\|\nabla \dot u_\e\|^2_{L^1(0,T;L^2(\O;\Ms))} \right).
\end{multline*}
According to Korn's inequality in $H^1_0(\O;\R^2)$, we have
$$\int_\O |\nabla \dot u_\e(t)-\nabla \dot w(t)|^2\, dx \leq 2 \int_\O |E \dot u_\e(t)-E \dot w(t)|^2\, dx,$$
hence
\begin{multline*}
 \|\varphi \partial_k^h \sigma_\e\|^2_{L^\infty(0,T;L^2(\O;\Ms))}+\e \|\varphi \partial_k^h \dot p_\e\|^2_{L^2(0,T;L^2(\O;\Ms))}\\
 \leq C\left(\|e_0\|^2_{H^1_{\rm loc}(\O;\Ms)} + \|f\|^2 _{L^\infty(0,T;H^1(\O;\R^2))}\|+\|E \dot u_\e\|^2_{L^2(0,T;L^2(\O;\Ms))} + \|\dot w\|^2_{L^2(0,T;H^1(\O;\R^2))} \right).
\end{multline*}
Thus, using \eqref{eq:first-estim} and \eqref{eq:second-estim}, we get that
$$ \|\varphi \partial_k^h \sigma_\e\|^2_{L^\infty(0,T;L^2(\O;\Ms))}+\e \|\varphi \partial_k^h \dot p_\e\|^2_{L^2(0,T;L^2(\O;\Ms))} \leq \frac{C}{\e}.$$
Letting $h \searrow 0$, we then obtain that for every open set $\O' \subset\subset \O$, there exists a constant $C>0$ (independent of $\e$) such that
\begin{equation}\label{eq:fourth-estim}
 \| \partial_k \sigma_\e\|^2_{L^\infty(0,T;L^2(\O';\Ms))}+\e \| \partial_k \dot p_\e\|^2_{L^2(0,T;L^2(\O';\Ms))} \leq \frac{C}{\e}.
\end{equation}
In particular, we get the following higher regularity properties:
\begin{equation}\label{eq:higher-reg1}
e_\e,\, \sigma_\e \in L^\infty(0,T;H^1_{\rm loc}(\O;\Ms)), \quad \dot p_\e \in L^2(0,T;H^1_{\rm loc}(\O;\Ms)).
\end{equation}

\subsubsection{Fifth energy estimate}

Using the additive decomposition, we have 
 $$E(\partial_k^h\dot u_\e)=\partial_k^h\dot e_\e + \partial_k^h\dot p_\e \quad \text{ in }\{x \in \O : \dist(x,\partial \O) >h\}\times (0,T).$$
 Let $\varphi \in \C^\infty_c(\O)$ be a cut-off function with $h < \dist({\rm supp \,} \varphi,\partial \O)$.  Taking the scalar product of the previous relation with $\varphi^2 (\partial_k^h \dot \sigma_\e-\partial_k^h \dot \chi)$ and integrating over $\O \times (0,T)$ yield
 \begin{multline*}
 2\int_0^T \Q \big(\varphi (\partial_k^h \dot e_\e(s)-\CC^{-1}\partial_k^h \dot \chi(s))\big)\, ds
 = \int_0^T \int_\O\varphi^2 (\partial_k^h \dot \sigma_\e-\partial_k^h\dot \chi):E(\partial_k^h\dot u_\e)\, dx \, ds\\
 -\int_0^T \int_\O\varphi^2 (\partial_k^h \dot \sigma_\e-\partial_k^h\dot \chi):\partial_k^h\dot p_\e\, dx \, ds- \int_0^T \int_\O\varphi^2 (\CC^{-1}\partial_k^h \dot \chi) :  (\partial_k^h \dot \sigma_\e-\partial_k^h \dot \chi)\, dx\, ds.
 \end{multline*}
We now look at the first integral on the right hand side.
Integrating by parts in space and using the equilibrium equations, together with the safe load condition \eqref{eq:safe},
yield
$$\int_0^T \int_\O\varphi^2 (\partial_k^h \dot \sigma_\e-\partial_k^h\dot \chi):E(\partial_k^h\dot u_\e)\, dx \, ds=\int_0^T \int_\O\varphi^2 (\partial_k^h \dot \sigma_\e-\partial_k^h\dot \chi):E(\partial_k^h\dot w)\, dx \, ds.$$
According to \eqref{eq:psi*}, $DH_\e^*$ is Lipschitz continuous with Lipschitz constant less than $1/\e$. Therefore, the flow rule $\dot p_\e=DH_\e^*(\sigma_\e)$ implies that 
$|\partial_k^h \dot p_\e|\leq |\partial_k^h\sigma_\e|/\e$ in $\O \times (0,T)$. Therefore, the Cauchy-Schwarz inequality, together with \eqref{eq:coerc-Q}, yields
\begin{multline*}
\| \varphi(\partial_k^h \dot \sigma_\e-\partial_k^h \dot \chi)\|_{L^2(0,T;L^2(\O;\Ms))}\\ \leq
C\left(\frac{1}{\e} \|\partial_k^h\sigma_\e\|_{L^2(0,T;L^2(\O;\Ms))} +\|w\|_{H^1([0,T];H^2(\O;\R^2))}+\|\chi\|_{H^1([0,T];H^1(\O;\Ms))} \right),\end{multline*}
where $C>0$ is a constant independent of $h$ and $\e$. Therefore using \eqref{eq:fourth-estim}, and passing to the limit as $h \to 0$, we get the following estimate: for every open set $\O'\subset\subset \O$, there exists a constant $C>0$ independent of $\e$ such that
$$\| \partial_k \dot e_\e\|_{L^2(0,T;L^2(\O';\Ms))} \leq\frac{C}{\e^2}.$$
As a consequence of this last estimate, together with \eqref{eq:higher-reg1}, the additive decomposition $Eu_\e=e_\e+p_\e$, and Korn's inequality, we deduce the following higher regularity properties:
$$e_\e, \sigma_\e \in H^1([0,T];H^1_{\rm loc}(\O;\Ms)),\quad \dot u_\e \in L^2(0,T;H^2_{\rm loc}(\O;\Ms)).$$

\subsection{A useful formula for the stress} 

Following the strategy of \cite{Se}, we introduce
 the function  $g : \Ms \to \R$ defined by
\begin{equation}\label{eq:g}
g(\xi):=\left\{
\begin{array}{ll}
\ds \frac{K_0}{2}({\rm tr}\xi)^2+\mu|\xi_D|^2 & \text{if } 2\mu|\xi_D| +2 \alpha K_0  {\rm tr}\xi \leq \kappa,\\
\ds \frac{\kappa\, {\rm tr}\xi}{2 \alpha} - \frac{\kappa^2}{8\alpha^2 K_0} + \Dc \left(|\xi_D|-\frac{{\rm tr} \xi }{2\alpha} + \frac{\kappa}{4\alpha^2 K_0} \right)^2_{\!\!+} & \text{if } 2\mu|\xi_D| + 2\alpha K_0  {\rm tr}\xi > \kappa,
\end{array}
\right.
\end{equation}
where
$$\Dc:=\left(\frac{1}{\mu}+\frac{1}{2\alpha^2 K_0} \right)^{-1}.$$
According to the results of \cite{ReSe,Se}, $g$ is the convex conjugate of the function $\tau \mapsto \frac12 \CC^{-1}\tau:\tau + I_K(\tau)$. Consequently, we have
$$g^*(\tau)= \frac12 \CC^{-1}\tau:\tau + I_K(\tau)\quad \text{ for all }\tau \in \Ms,$$
and, in particular, $g$ is convex, lower semicontinuous, and locally Lipschitz continuous. One can  check that $g$ is actually of class $\C^1$, and that its differential is given by
$$Dg(\xi) \!=\!
\begin{cases}
\!\ds K_0({\rm tr}\xi) {\rm Id} +2\mu \xi_D &\text{if } 2\mu|\xi_D| + 2\alpha K_0  {\rm tr}\xi \leq \kappa,\smallskip\\
\!\ds  \frac{\kappa}{2\alpha} {\rm Id}\!+\! 2\Dc\!\! \left(|\xi_D|-\frac{{\rm tr}\xi}{2\alpha} + \frac{\kappa}{4\alpha^2 K_0} \right)_{\!\!+}\left(\frac{\xi_D}{|\xi_D|}-\frac{1}{2\alpha}{\rm Id}  \right) &\text{if } 2\mu|\xi_D| + 2\alpha K_0  {\rm tr}\xi > \kappa
\end{cases}$$
for every $\xi \in \Ms$.
In addition, there exists a constant $a_0>0$ (only depending on $\alpha$, $\kappa$, $K_0$) such that
\begin{equation}\label{eq:coerc-g}
g(\xi)  \geq  a_0 |\xi|^2 ,\quad  \text{ if } 2\mu|\xi_D| +2\alpha K_0  {\rm tr}\xi \leq \kappa.
\end{equation}

According to convex analysis,  the regularized flow rule \eqref{eq:reg-flow-rule} can be equivalently written, for a.e.\ $t \in [0,T]$, as
$$\dot p_\e(t) \in \partial I_K\big(\sigma_\e(t) - \e \dot p_\e(t)\big)\quad \text{ in }\O,$$
or still
$$\dot p_\e(t)  + e_\e(t) - \e \CC^{-1}\dot p_\e(t) \in  \partial g^*(\sigma_\e(t) - \e \dot p_\e(t))\quad \text{ in }\O.$$
Again by convex duality we can rewrite the regularized flow rule in terms of the following formula, 
which will be instrumental in the subsequent analysis: for a.e.\ $t \in [0,T]$,
\begin{equation}\label{form-sigma}
\sigma_\e(t) - \e \dot p_\e(t)=Dg\big(\dot p_\e(t)  + e_\e(t) - \e \CC^{-1}\dot p_\e(t)\big)\quad \text{ in }\O.
\end{equation}

\section{A chain rule formula for the weak derivatives of the stress}\label{sec:4}

In order to prove Theorem~\ref{thm:main}, we will show that $\sigma^\e \in L^\infty(0,T;H^1_{\rm loc}(\O;\Ms))$ with a bound on its norm that is uniform with respect to $\e$. To do that, it will be useful to differentiate formula \eqref{form-sigma}
with respect to the space variables. Unfortunately, the function $g$ is only of class $\mathcal C^1$, with Lipschitz continuous partial derivatives, so that the classical chain rule formula for vector-valued Sobolev functions does not apply. To overcome this issue, we will use the special form of $g$ given in \eqref{eq:g}, together with the generalized chain rule formula established in \cite[Corollary~3.2]{ADM}. 

\subsection{Some geometrical preliminaries}

We first establish several technical results which will be of use in the proof of that result. Let us introduce some notation: let $G :\Ms \to \Ms$ be defined, for all $\xi \in \Ms$, by
$$G(\xi):=Dg\left(\xi+\frac{\kappa}{4\alpha K_0}{\rm Id}\right).$$
We define the open sets
\begin{eqnarray*}
\A_1 & := & \{\xi \in \Ms : \mu |\xi_D| + \alpha K_0 {\rm tr}\xi<0\},\\
\A_2 & := & \left\{\xi \in \Ms : |\xi_D|< \frac{{\rm tr}\xi}{2\alpha} \right\},\\
\A_3 & := & \Ms \setminus (\overline \A_1 \cup \overline \A_2).
\end{eqnarray*}
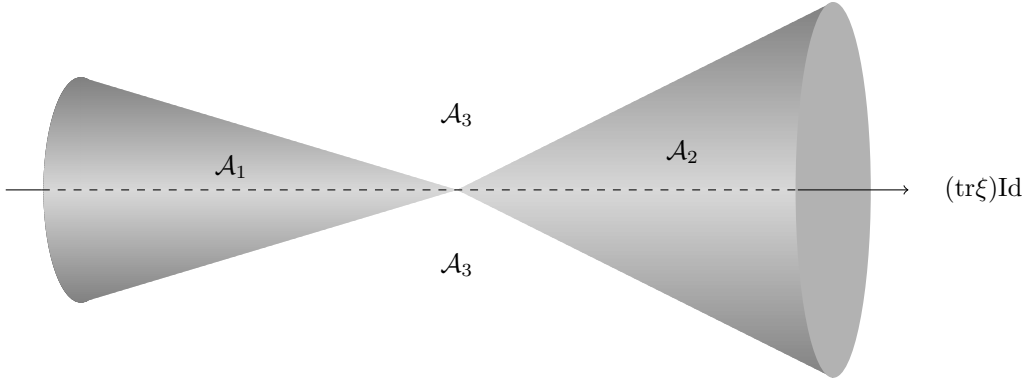
\begin{figure}[!h]\begin{tikzpicture}
\centering
	\fill[top color=black!50,bottom color=black!50,middle color=gray!30] (0,1.5) ellipse (.5 and 1.5);
	\shade[top color=black!50,bottom color=black!50,middle color=gray!30]  (0,0) -- (0,3) -- (5,1.5) -- cycle;

	\shade [top color=black!50,bottom color=black!50,middle color=gray!30]  (10,-1) -- (10,4) -- (5,1.5) -- cycle;
	\fill[color=black!30] (10,1.5) ellipse (.5 and 2.5);
	
	\node (8,2) at (8,2) {$\mathcal A_2$};
	\node (2,1.8) at (2,1.8) {$\mathcal A_1$};
	\node (5,0.5) at (5,0.5) {$\mathcal A_3$};
	\node (5,2.5) at (5,2.5) {$\mathcal A_3$};
	\node (12,1.5) at (12,1.5) {$({\rm tr}\xi) {\rm Id}$};
	
	\draw [-] (-1,1.5) -- (-0.52,1.5);
	\draw [dashed] (-0.52,1.5) -- (9.53,1.5) ;
	\draw [->] (9.53,1.5) -- (11,1.5);
\end{tikzpicture}
\caption{Representation of the sets $\mathcal A_1$, $\mathcal A_2$ and $\mathcal A_3$.}
\label{fig:1}
\end{figure}

\noindent
The sets $\A_1$ and $\A_2$ are open cones in $\Ms$ (which is a three-dimensional space) with same vertex at the origin, $\partial \A_1 \cap \partial \A_2=\{0\}$, and axis given by the hydrostatic matrices (see Figure \ref{fig:1}). Note, in particular, that if $\xi \in \A_3$, then $-\frac{\mu}{\alpha K_0} |\xi_D|<{\rm tr}\xi<2 \alpha |\xi_D|$, so that
\begin{equation}\label{eq:xiA3}
|\xi|\leq c|\xi_D| \quad \text{ for all }\xi \in \A_3,
\end{equation}
for some constant $c>0$ only depending on $\mu$, $\alpha$, and $K_0$. It follows that
\begin{equation}\label{eq:G}
G(\xi) = \frac{\kappa}{2\alpha} {\rm Id}+
\begin{cases}
\ds K_0({\rm tr}\xi) {\rm Id} +2\mu \xi_D & \text{ if } \xi \in  \A_1,\\
\ds 0 &  \text{ if } \xi \in  \A_2 \cup \{0\},\\
\ds  2\Dc \left(|\xi_D|-\frac{{\rm tr}\xi}{2\alpha}  \right)\left(\frac{\xi_D}{|\xi_D|}-\frac{1}{2\alpha}{\rm Id}  \right) & \text{ if } \xi \in \overline \A_3 \setminus \{0\}.
\end{cases}
\end{equation}
Note that $G$ is Lipschitz continuous on $\Ms$, and its restriction to the open sets $\A_1$, $\A_2$, and $\A_3$ is smooth. The only non differentiability points lie at the interfaces $\partial \A_1 \cup \partial \A_2$. Let us also define, for all $\xi$ and $\eta \in \Ms$,
$$\mathcal G(\xi,\eta) =
\begin{cases}
\ds K_0({\rm tr }\, \eta) {\rm Id} +2\mu \eta_D  & \text{ if } \xi \in  \A_1,\\
\ds 0 &  \text{ if }  \xi \in  \A_2\cup \{0\},\\
\ds  2\Dc \left(\frac{\xi_D:\eta_D}{|\xi_D|}-\frac{{\rm tr}\eta}{2\alpha}  \right)\left(\frac{\xi_D}{|\xi_D|}-\frac{1}{2\alpha}{\rm Id}  \right)&
\smallskip\\
\qquad\qquad \ds +2\Dc \left(|\xi_D| - \frac{{\rm tr}\xi}{2\alpha} \right) \frac{1}{| \xi_D|}\left(\eta_D- \frac{\xi_D : \eta_D}{|\xi_D|^2}\xi_D \right) &  \text{ if }   \xi \in \overline \A_3\setminus \{0\}.
\end{cases}$$
Note that if $\xi \in \A_1 \cup \A_2 \cup \A_3$, then 
$$\mathcal G(\xi,\eta)=DG(\xi)\eta=D^2g\left(\xi+\frac{\kappa}{4\alpha  K_0}{\rm Id}\right)\eta.$$

\begin{rem}\label{rem:CS}
Since, for each $\xi \in \Ms$, the bilinear form $(\eta_1,\eta_2) \mapsto \mathcal G(\xi,\eta_1):\eta_2$ is symmetric and non-negative, if follows from the Cauchy-Schwarz inequality (applied to this bilinear form) that
$$\mathcal G(\xi,\eta_1):\eta_2 \leq \sqrt{\mathcal G(\xi,\eta_1):\eta_1}\sqrt{\mathcal G(\xi,\eta_2):\eta_2}$$
for every $\xi,\eta_1,\eta_2 \in \Ms$.
In addition, there exists a constant $C_*>0$, depending only on $K_0$, $\mu$, $\alpha$, and $\kappa$, such that 
$$\mathcal G(\xi,\eta):\eta \leq C_* |\eta|^2$$
for all $\xi,\eta \in \Ms$.
\end{rem}

\begin{lem}\label{lem:normal}
Let $\sigma \in (\partial \A_1 \cup \partial \A_2) \setminus \{0\}$, and $\tau_1, \tau_2 \in \Ms$. We define the set 
$$\mathcal T:=\left\{\sigma +z_1 \tau_1 +  z_2\tau_2 :\ z \in \R^2 \right\}.$$
Assume that $G|_{\mathcal T}$ is differentiable at $\sigma$. Then, 
$$\tau_1 : \nu(\sigma)=\tau_2 : \nu(\sigma)=0,$$
where 
$$\nu(\sigma):=
\begin{cases}
\mu\dfrac{\sigma_D}{|\sigma_D|}+\alpha  K_0 {\rm Id} & \text{ if } \sigma \in \partial \A_1,\smallskip\\
\dfrac{\sigma_D}{|\sigma_D|}-\dfrac{1}{2\alpha}{\rm Id}& \text{ if } \sigma \in \partial \A_2
\end{cases}$$
is normal to $\partial \A_1 \cup \partial \A_2$ at $\sigma$ and oriented from $\A_1 \cup \A_2$ to $\A_3$.
\end{lem}

\begin{proof}
Suppose first that $\sigma \in \partial \A_2$, so that 
\begin{equation}\label{bdry}
|\sigma_D|= \frac{{\rm tr}\sigma}{2\alpha}.
\end{equation}
Assume by contradiction that $\tau_{k_0} : \nu(\sigma) \neq 0$ for some $k_0 \in \{1,2\}$, {\it i.e.},
\begin{equation}\label{normal}
\frac{\sigma_D:(\tau_{k_0})_D}{|\sigma_D|} - \frac{{\rm tr}\tau_{k_0}}{2\alpha}\neq 0.
\end{equation}
Then, for $t \in \R$ small enough so that $\sigma + t\tau_{k_0}\neq 0$, we can assume that $\sigma +  t\tau_{k_0} \in \A_2$  if $t<0$, while $\sigma +  t\tau_{k_0} \in \A_3$  if $t>0$. Since the restriction of $G$ to $\A_2$ is constant, and $\{ \sigma+t\tau_{k_0}: t \in \R\}$ is a direction in $\mathcal T$ along which (by assumption) $G$ is differentiable at $\sigma$, we deduce that
\begin{equation}\label{deriveg}
\frac{\partial}{\partial \tau_{k_0}}[G|_{\mathcal T}](\sigma)=\lim_{t \to 0^-}\frac{G(\sigma +t\tau_{k_0})-G(\sigma)}{t}=0.
\end{equation}
Similarly, we use the differentiability of $G$ in $\A_3$ to get that
\begin{multline}\label{derived}
\frac{\partial}{\partial \tau_{k_0}}[G|_{\mathcal T}](\sigma)=\lim_{t \to 0^+}\frac{G(\sigma +t\tau_{k_0})-G(\sigma)}{t}\\
=2\Dc\left( \frac{\sigma_D:(\tau_{k_0})_D}{|\sigma_D|} - \frac{{\rm tr}\tau_{k_0}}{2\alpha}\right)  \left( \frac{\sigma_D}{|\sigma_D|}- \frac{1}{2\alpha} {\rm Id} \right) \\
+ 2\Dc \left( |\sigma_D|-\frac{{\rm tr}\sigma}{2\alpha}  \right)\frac{1}{|\sigma_D|} \left((\tau_{k_0})_D - \frac{\sigma_D:(\tau_{k_0})_D}{|\sigma_D|^2}\sigma_D \right).
\end{multline}
According to \eqref{bdry}, the second term in the right hand side of the previous identity vanishes. On the other hand, the combination of \eqref{deriveg}, \eqref{derived}, and the assumption \eqref{normal} yield $ \frac{\sigma_D}{|\sigma_D|}= \frac{1}{2\alpha} {\rm Id}$, which is impossible.

Assume now that $\sigma \in \partial \A_1$, and, again by contradiction, that $\tau_{k_0} : \nu(\sigma) \neq 0$ for some $k_0 \in \{1,2\}$, {\it i.e.},
\begin{equation}\label{normal2}
\mu \frac{\sigma_D:(\tau_{k_0})_D}{|\sigma_D|}+\alpha K_0 ({\rm tr}\tau_{k_0}) \neq 0.
\end{equation}
Then, for $t \in \R$ small enough so that $\sigma + t\tau_{k_0}\neq 0$, we can assume that $\sigma +  t\tau_{k_0} \in \A_1$  if $t<0$, while $\sigma +  t\tau_{k_0} \in \A_3$  if $t>0$. An analogous argument as before shows that
\begin{equation}\label{deriveg2}
\frac{\partial}{\partial \tau_{k_0}}[G|_{\mathcal T}](\sigma)=\lim_{t \to 0^-}\frac{G(\sigma +t\tau_{k_0})-G(\sigma)}{t}=K_0 ({\rm tr} \tau_{k_0}) {\rm Id} +2\mu (\tau_{k_0})_D,
\end{equation}
and
\begin{multline}\label{derived2}
\frac{\partial}{\partial \tau_{k_0}}[G|_{\mathcal T}](\sigma)=\lim_{t \to 0^+}\frac{G(\sigma +t\tau_{k_0})-G(\sigma)}{t}\\
=2\Dc\left( \frac{\sigma_D:(\tau_{k_0})_D}{|\sigma_D|} - \frac{{\rm tr}\tau_{k_0}}{2\alpha}\right)  \left( \frac{\sigma_D}{|\sigma_D|}- \frac{1}{2\alpha} {\rm Id} \right) \\
+ 2\Dc \left( |\sigma_D|-\frac{{\rm tr}\sigma}{2\alpha}  \right)\frac{1}{|\sigma_D|} \left((\tau_{k_0})_D - \frac{\sigma_D:(\tau_{k_0})_D}{|\sigma_D|^2}\sigma_D \right).
\end{multline}
Equating \eqref{deriveg2} and \eqref{derived2}, and taking the trace yield
$$ \left(\frac{1}{\mu}+\frac{1}{2\alpha^2 K_0} \right)  K_0 {\rm tr} \tau_{k_0}=-\frac{1}{\alpha}\left( \frac{\sigma_D:(\tau_{k_0})_D}{|\sigma_D|} - \frac{{\rm tr}\tau_{k_0}}{2\alpha}\right),$$
which is against \eqref{normal2}.
\end{proof}

Let us denote by $T_\sigma(\partial \A_1)$ (resp. $T_\sigma(\partial \A_2)$)  the tangent space to $\partial \A_1$ (resp. $\partial \A_2$) at $\sigma$.

\begin{lem}\label{lem:tangent}
For every $\sigma \in\partial \A_i \setminus \{0\}$ ($i=1,2$),  there exists some $\rho>0$ such that
$$T_\sigma(\partial \A_i) \cap B(\sigma,\rho) \subset \overline \A_3.$$
\end{lem}

\begin{proof}
Let us check this property for $i=2$. Let $\sigma\in \partial \A_2 \setminus \{0\}$. Since $\sigma \neq 0$, there exists some $\rho>0$ such that $B(\sigma,\rho) \cap \overline{\A_1}=\emptyset$. We define the convex function
$$a(\tau):=  |\tau_D| - \frac{1}{2\alpha}{\rm tr}\tau$$
for every $\tau \in \Ms$.
Since $\sigma \neq 0$, the function $a$ is differentiable at $\sigma$ with gradient equal to $\nu(\sigma)$, and for every $\tau \in T_\sigma(\partial \A_2)\cap B(\sigma,\rho)$, we have
$$a(\tau) \geq a(\sigma) + \nu(\sigma):(\tau - \sigma).$$
Since $\sigma  \in \partial \A_2$, we have that $a(\sigma)=0$, while the fact that $\tau - \sigma \in T_\sigma(\partial \A_2)$ yields $\nu(\sigma):(\tau - \sigma)=0$. Therefore
$a(\tau) \geq 0$, that is, $\tau \not\in \A_2$, and since $\tau \not\in \overline \A_1$ neither, then $\tau \in \overline \A_3$ (see Figure \ref{fig:2}).
\end{proof}

\begin{figure}[h!]
  \begin{tikzpicture}[fill opacity=.6]
	\centering
	\fill[top color=black!50,bottom color=black!50,middle color=gray!30,opacity=1] (0,1.5) ellipse (.5 and 1.5);
	\shade[top color=black!50,bottom color=black!50,middle color=gray!30,opacity=1]  (0,0) -- (0,3) -- (5,1.5) -- cycle;

	\shade [top color=black!50,bottom color=black!50,middle color=gray!30,opacity=1]  (10,-1) -- (10,4) -- (5,1.5) -- cycle;
	\fill[color=black!30,opacity=1] (10,1.5) ellipse (.5 and 2.5);

  	\fill [color=gray!30] (-1,4) -- (7,1.6) -- (11,-1.6) --  (3,0.6) -- cycle;
	\node (-1,4.2) at (-1,4.2) {$T_\sigma(\partial \mathcal A_1)$};
	\shade[ball color = white] (2.5,2.22) circle (0.5);
	\draw[dashed,rotate around={-15:(2.5,2.22)}] (2.5,2.22) ellipse (0.5 and 0.1);
	
	\node (8.5,1.5) at (8.5,1.5) {$\mathcal A_2$};
	\node (1,1.5) at (1,1.5) {$\mathcal A_1$};
	\node (5,3) at (5,3) {$\mathcal A_3$};
	\node (2.5,2.53) at (2.5,2.53) {{\small $\sigma$}};
	\node (2.5,2.22) at (2.5,2.22) {{\scriptsize $\bullet$}};
	\node (3.6,2.3) at (3.6,2.3) {{\small $B(\sigma,\rho)$}};

\end{tikzpicture}
\caption{The tangent space $T_\sigma(\partial \mathcal A_1)$ of $\partial \mathcal A_1$ at $\sigma \in \partial \mathcal A_1 \setminus \{0\}$ might intersect $\mathcal A_2$, but not in a small neighborhood $B(\sigma,\rho)$ of $\sigma$.}
\label{fig:2}
\end{figure}
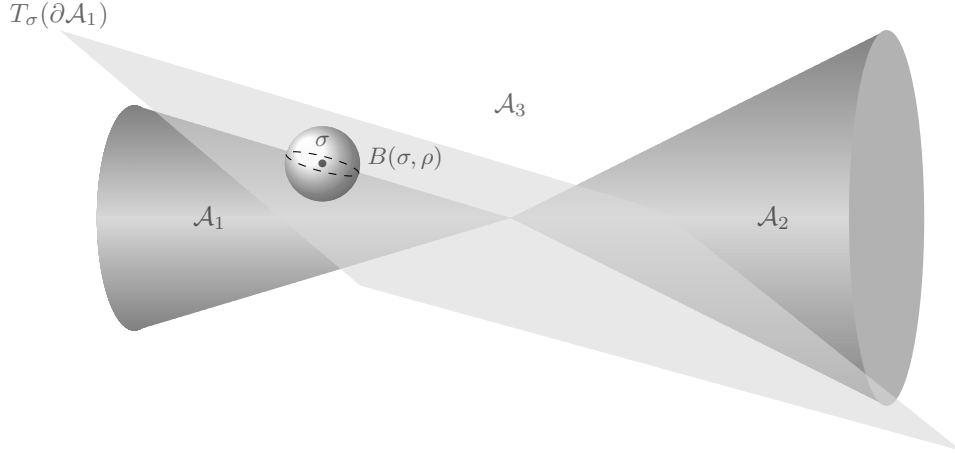

\subsection{Proof of the chain rule formula}

As a consequence of the previous results, we get that a chain rule type formula holds.

\begin{cor}\label{cor:form-sigma}
Let us set $\xi^\e:=\dot p_\e  + e_\e - \e \CC^{-1}\dot p_\e-\frac{\kappa}{4\alpha K_0}{\rm Id}$. Then for all $k \in \{1,2\}$,
$$\partial_k \sigma_\e -\e \partial_k \dot p_\e= \mathcal G(\xi^\e,\partial_k\xi^\e)=
\begin{cases}
\ds K_0{\rm tr}(\partial_k \xi^\e) {\rm Id} +2\mu \partial_k \xi^\e_D  & \text{ a.e.\ in } A^{\xi^\e}_1,\\
\ds 0 & \text{ a.e.\ in }  A^{\xi^\e}_2,\\
\ds  2\Dc \left(\frac{\xi^\e_D:\partial_k\xi^\e_D}{|\xi^\e_D|}-\frac{{\rm tr}(\partial_k \xi^\e)}{2\alpha}  \right)\left(\frac{\xi^\e_D}{|\xi^\e_D|}-\frac{1}{2\alpha}{\rm Id}  \right)\smallskip\\
\quad \ds +2\Dc \left(|\xi^\e_D| - \frac{{\rm tr}\xi^\e}{2\alpha} \right) \frac{1}{| \xi^\e_D|}\left(\partial_k \xi^\e_D- \frac{\xi^\e_D : \partial_k\xi^\e_D}{|\xi^\e_D|^2}\xi^\e_D \right)& \text{ a.e.\ in }  A^{\xi^\e}_3,
\end{cases}$$
where
\begin{eqnarray*}
A_1^{\xi^\e}&:=&\{(x,t) \in \O \times (0,T) : \xi^\e(x,t) \in \A_1\},\\
A_2^{\xi^\e}&:=&\{(x,t) \in \O \times (0,T) : \xi^\e(x,t) \in \A_2 \cup \{0\} \},\\
A_3^{\xi^\e}&:=&\{(x,t) \in \O \times (0,T): \xi^\e(x,t) \in \overline \A_3\setminus \{0\}\}.
\end{eqnarray*}
\end{cor}

\begin{proof}
 In order not to overburden notation, we omit the subscript $\e$. Using the regularity properties \eqref{eq:higher-reg1}, we deduce that $\xi \in L^2(0,T;H^1_{\rm loc}(\O;\Ms))$, and by \eqref{form-sigma}, we get $\sigma-\e \dot p=G \circ \xi$.
According to the generalized chain rule formula established in \cite[Corollary~3.2]{ADM}, if we denote
the affine space passing through $\xi(x,t)$ by
$$\mathcal T_{(x,t)}^{\xi}:=\left\{ \xi(x,t) + \partial_1\xi(x,t) z_1+\partial_2\xi(x,t) z_2:\ z \in \R^2\right\},$$
then we have that $G|_{\mathcal T_{(x,t)}^{\xi}}$ is differentiable at $\xi(x,t)$ for a.e.\ $(x,t) \in \O\times (0,T)$, and for all $1 \leq i,j \leq 2$,
\begin{equation}\label{eq:chain-rule}
\nabla \sigma_{ij}-\e\nabla \dot p_{ij}=\nabla\big(G_{ij}|_{\mathcal T_{(x,t)}^{\xi}}\big)(\xi)\nabla \xi.
\end{equation}

First, since $\nabla \xi=0$ a.e.\ in $\{\xi=0\}$, it results from formula \eqref{eq:chain-rule} that $\nabla \sigma-\e\nabla \dot p=0$ a.e.\ in $\{\xi=0\}$, and as $\{\xi=0\}\subset A_2^{\xi}$, we deduce that
\begin{equation}\label{eq:xi=0}
\partial_k \sigma-\e \partial_k\dot p=\mathcal G(\xi,\partial_k \xi) \quad \text{ a.e.\ in }\{\xi=0\}.
\end{equation}

Secondly, if $\xi(x,t) \in \A_i$ for some $i \in \{1,2,3\}$, then $\xi(x,t)$ is a differentiability point of $G$, and using again \eqref{eq:chain-rule}, 
\begin{equation}\label{eq:xiAi}
\partial_k \sigma-\e\partial_k \dot p= DG(\xi):\partial_k \xi=\mathcal G(\xi,\partial_k\xi) \quad \text{ a.e.\ in }\bigcup_{i=1}^3 \{\xi \in \A_i\}.
\end{equation}

It thus remains to check the formula when $\xi(x,t) \in \partial \A_3 \setminus \{0\}=(\partial \A_1 \cup \partial \A_2) \setminus \{0\}$. According to Lemma~\ref{lem:normal}, we deduce that for a.e.\ $(x,t)$ in that set,
$$\partial_k\xi(x,t) : \nu(\xi(x,t))=0 \quad \text{for all }k \in \{1,2\},$$
and thus, $\mathcal T_{(x,t)}^{\xi} \subset T_{\xi(x,t)}(\partial \A_1) \cup T_{\xi(x,t)}(\partial \A_2)$. Using next Lemma~\ref{lem:tangent} yields that 
$$\mathcal T_{(x,t)}^{\xi}\cap B(\xi(x,t),\rho(x,t)) \subset \overline \A_3$$
for some $\rho(x,t)>0$. Thus, according to the expression \eqref{eq:G} of $G$ in $\overline\A_3$ together with \eqref{eq:chain-rule},
\begin{equation}\label{eq:xidA3}
\partial_k \sigma-\e\partial_k\dot p=\mathcal G(\xi,\partial_k\xi)\quad \text{ a.e.\ in }\{\xi \in \partial \A_3 \setminus \{0\}\}.
\end{equation}
Gathering \eqref{eq:xi=0}, \eqref{eq:xiAi}, and \eqref{eq:xidA3} leads to the desired expression for the derivatives of $\sigma$.
\end{proof}

\section{A uniform bound for the Sobolev norm of the stress}\label{sec:5}

We next show that the $L^\infty(0,T;H^1_{\rm loc}(\O;\Ms))$ norm of $\sigma_\e$ can be uniformly controled with respect to $\e$.

\begin{prop}\label{prop:unif-bound}
Under the same assumptions of Theorem~\ref{thm:main}, we have that, for every open set $\O' \subset\subset\O$, 
$$\sup_{\e>0} \| \sigma_\e\|_{L^\infty(0,T;H^1(\O';\Ms))}<+\infty.$$
\end{prop}

\begin{proof}
All constants appearing in this proof are independent of $\e$. In order to lighten the notation, we do not explicitely write  the various $dx$, $ds$ inside the integrals. As in Corollary~\ref{cor:form-sigma}, we also omit the subscript $\e$ and we write 
$$\xi:=\dot p  + e - \e \CC^{-1}\dot p-\frac{\kappa}{4\alpha K_0}{\rm Id} \in  L^2(0,T;H^1_{\rm loc}(\O;\Ms)).$$
We start with the equilibrium equation
\begin{equation}\label{eq:div}
-{\rm div} \sigma= f\text{ in }\O \times (0,T).
\end{equation}
Let $\O' \subset \subset\O''\subset\subset \O$ and $\varphi \in \C^\infty_c(\O;[0,1])$ be a cut-off function satisfying $\varphi \equiv 1$ on $\O'$ and $\varphi\equiv 0$ on $\O \setminus \overline{\O''}$. By Proposition~\ref{prop:higher-reg},  if $t \in [0,T]$ is arbitrary, we have $\varphi^6 (\partial_k \dot u) \chi_{[0,t]} \in L^2(0,T;H^1_0(\O;\R^2))$. We then take the partial derivative of \eqref{eq:div} with respect to $x_k$ ($k=1,2$), and use the previous function as test function. Denoting by $Q_t:=\O \times (0,t)$ the space-time cylinder, we get
\begin{equation}\label{MGadd2}
\iint_{Q_t}  \partial_k \sigma : (E\partial_k \dot u) \varphi^6+\iint_{Q_t}\partial_k \sigma :(\partial_k \dot u \otimes \nabla \varphi^6)= \iint_{Q_t} \partial_k f \cdot (\partial_k\dot u) \varphi^6,
\end{equation}
where, from now on, we use the summation convention over repeated indexes.
Then, writing $\partial_k \dot u_j=2E_{kj}\dot u - \partial_j \dot u_k$, we find
$$
\iint_{Q_t}\partial_k \sigma :(\partial_k \dot u \otimes \nabla \varphi^6)
= 2 \iint_{Q_t} \partial_k  \sigma_{ij} (E_{kj} \dot u) \partial_i \varphi^6
 - \iint_{Q_t}  \partial_k  \sigma_{ij} \,\partial_j \dot u_k\,\partial_i \varphi^6.
$$
An integration by parts in the last integral at the right hand side, together with \eqref{eq:div}, yields
$$\iint_{Q_t} \partial_k  \sigma_{ij}\,\partial_j \dot u_k\,\partial_i \varphi^6= \iint_{Q_t} \partial_k f_i \dot u_k \partial_i \varphi^6- \iint_{Q_t} \dot u_k\, \partial_k  \sigma_{ij} \, \partial_j\partial_i \varphi^6,$$
and thus, using the previous identities in \eqref{MGadd2},
\begin{multline}\label{eq:stim1}
\iint_{Q_t}  \partial_k \sigma: (E\partial_k \dot u) \varphi^6
=- 2 \iint_{Q_t} \partial_k  \sigma_{ij} (E_{kj} \dot u) \partial_i \varphi^6
- \iint_{Q_t} \dot u_k\, \partial_k  \sigma_{ij} \, \partial_j\partial_i \varphi^6\\
 +\iint_{Q_t} \left(-\Delta f \cdot \dot u \varphi^6 -\partial_k f \cdot \dot u \partial_k \varphi^6  + \partial_k f_i \dot u_k \partial_i \varphi^6\right).
\end{multline}
Using the additive decomposition and the definition of $\xi$, we have
\begin{equation}\label{eq:useful-form}
E\dot u =\dot e +\dot p, \quad \dot p= \xi -e+\e\CC^{-1}\dot p + \frac{\kappa}{4\alpha K_0}{\rm Id}.
\end{equation}
Hence 
\begin{multline*}
\iint_{Q_t}  \partial_k \sigma: (E\partial_k \dot u) \varphi^6 =\iint_{Q_t}  \partial_k \sigma: \partial_k \dot e\, \varphi^6 + \e\iint_{Q_t}  \partial_k \dot p: \partial_k \dot p\, \varphi^6\\
+\iint_{Q_t}  (\partial_k \sigma-\e\partial_k\dot p): (\partial_k \xi - \partial_k e+\e\CC^{-1} \partial_k \dot p) \varphi^6,
\end{multline*}
and inserting inside \eqref{eq:stim1} yields
\begin{multline}\label{eq:stim10}
\iint_{Q_t}  \partial_k \sigma: \partial_k \dot e\, \varphi^6 + \e\iint_{Q_t} \partial_k \dot p: \partial_k \dot p\,  \varphi^6+ \iint_{Q_t}(\partial_k \sigma-\e\partial_k\dot p): \partial_k \xi\, \varphi^6\\
=- 2 \iint_{Q_t} \partial_k  \sigma_{ij} \, (E_{kj} \dot u)\, \partial_i \varphi^6
- \iint_{Q_t}  \dot u_k \,\partial_k  \sigma_{ij} \, \partial_j\partial_i \varphi^6\\
+\iint_{Q_t}   \CC^{-1}(\partial_k \sigma-\e\partial_k\dot p):(\partial_k \sigma-\e\partial_k\dot p)\varphi^6\\
+\iint_{Q_t} \left(-\Delta f \cdot \dot u \varphi^6 -\partial_k f \cdot \dot u \partial_k \varphi^6  + \partial_k f_i \dot u_k \partial_i \varphi^6\right).
\end{multline}

\medskip

{\bf Step 1.} We first bound from below the left hand side of \eqref{eq:stim10}. First, by Corollary~\ref{cor:form-sigma}
we have that $\partial_k \sigma-\e\partial_k\dot p =\mathcal G(\xi,\partial_k \xi)$, and thus
\begin{equation}\label{eq:sigmakxik}
(\partial_k \sigma-\e\partial_k\dot p):\partial_k\xi = 
\begin{cases}
\ds\sum_{k=1}^2 K_0({\rm tr}\partial_k \xi)^2 +2\mu |\partial_k \xi_D|^2  & \text{a.e.\ in } A^{\xi}_1,\smallskip\\
\ds 0 & \text{a.e.\ in }  A^{\xi}_2,\smallskip\\
\ds\sum_{k=1}^2  2\Dc \left(\frac{\xi_D:\partial_k\xi_D}{|\xi_D|}-\frac{{\rm tr}\partial_k \xi}{2\alpha}  \right)^2\smallskip\\
\hspace{0.2cm} \ds + \sum_{k=1}^22\Dc \left(|\xi_D| - \frac{{\rm tr}\xi}{2\alpha} \right) \frac{1}{| \xi_D|}\left(|\partial_k \xi_D|^2- \frac{(\xi_D : \partial_k\xi_D)^2}{|\xi_D|^2} \right) & \text{a.e.\ in }  A^{\xi}_3.
\end{cases}
\end{equation}
On the other hand, according to Remark~\ref{rem:CS}, we get that
\begin{multline*}
(\partial_k \sigma-\e\partial_k\dot p):(\partial_k \sigma-\e\partial_k\dot p)=\mathcal G(\xi,\partial_k \xi):(\partial_k \sigma-\e\partial_k\dot p)\\
 \leq \sqrt{\mathcal G(\xi,\partial_k \xi):\partial_k \xi}\sqrt{\mathcal G(\xi,\partial_k \sigma-\e\partial_k\dot p):(\partial_k \sigma-\e\partial_k\dot p)}\\
 \leq \sqrt{(\partial_k \sigma-\e\partial_k\dot p):\partial_k \xi}\sqrt{ C_*(\partial_k \sigma-\e\partial_k\dot p):(\partial_k \sigma-\e\partial_k\dot p)},
\end{multline*}
so that
$$(\partial_k \sigma-\e\partial_k\dot p):(\partial_k \sigma-\e\partial_k\dot p) \leq C_*(\partial_k \sigma-\e\partial_k\dot p): \partial_k \xi.$$
Thus, applying this inequality in \eqref{eq:stim10}, we get that
\begin{multline}\label{eq:stim11}
\iint_{Q_t}  \partial_k \sigma: \partial_k \dot e\, \varphi^6 + \e\iint_{Q_t}  \partial_k \dot p: \partial_k \dot p\,  \varphi^6\\
+\frac12\iint_{Q_t}  (\partial_k \sigma-\e\partial_k\dot p): \partial_k \xi\, \varphi^6+\frac{1}{2 C_*}\iint_{Q_t}(\partial_k \sigma-\e\partial_k\dot p):(\partial_k \sigma-\e\partial_k\dot p)\varphi^6\\
\leq - 2 \iint_{Q_t} \partial_k  \sigma_{ij} \, (E_{kj} \dot u)\, \partial_i \varphi^6
-  \iint_{Q_t} \dot u_k\, \partial_k  \sigma_{ij}\, \partial_j\partial_i \varphi^6\\
+\iint_{Q_t}   \CC^{-1}(\partial_k \sigma-\e\partial_k\dot p):(\partial_k \sigma-\e\partial_k\dot p)\varphi^6\\
+\iint_{Q_t} \left(-\Delta f \cdot \dot u \varphi^6 -\partial_k f \cdot \dot u \partial_k \varphi^6  + \partial_k f_i \dot u_k \partial_i \varphi^6\right).
\end{multline}

\medskip

{\bf Step 2.} We now bound from above the right hand side of \eqref{eq:stim11}. Since $\partial_i \varphi^6=\varphi^3 (6 \varphi^2 \partial_i \varphi)$, and $\partial_j\partial_i \varphi^6=\varphi^3(30\varphi\partial_j \varphi \partial_i \varphi + 6 \varphi^2 \partial_j\partial_i\varphi)$, we deduce from the Cauchy-Schwarz and Young inequalities that
$$
\iint_{Q_t} \left(-\Delta f \cdot \dot u \varphi^6 -\partial_k f \cdot \dot u \partial_k \varphi^6  + \partial_k f_i \dot u_k \partial_i \varphi^6\right)
\leq C \left(\|f\|^2_{L^2(0,T;H^2(\O;\R^2))} + \|\dot u\|^2_{L^2(Q_T;\R^2)} \right),
$$
$$-\iint_{Q_t} \dot u_k \, \partial_k  \sigma_{ij} \, \partial_j\partial_i \varphi^6  \leq C\left(\frac12 \iint_{Q_t}  \partial_k  \sigma:\partial_k  e\, \varphi^6\, dx\, ds + \|\dot u\|^2_{L^2(Q_T;\R^2)}\right),$$
and
$$\iint_{Q_t}   \CC^{-1}(\partial_k \sigma-\e\partial_k\dot p):(\partial_k \sigma-\e\partial_k\dot p)\varphi^6
\leq C \left(\frac12 \iint_{Q_t}  \partial_k  \sigma : \partial_k  e\, \varphi^6+\e^2\iint_{Q_t} \partial_k \dot p: \partial_k \dot p\, \varphi^6\right).$$
If $C\e<1/2$, the last integral can be absorbed by the second one at the left hand side of \eqref{eq:stim11}. We then replace in \eqref{eq:stim11} to get
\begin{multline}\label{eq:stim111}
\iint_{Q_t}  \partial_k \sigma : \partial_k \dot e\, \varphi^6 + \frac{\e}{2} \iint_{Q_t}  \partial_k \dot p: \partial_k \dot p\,  \varphi^6\\
+\frac12\iint_{Q_t} (\partial_k \sigma-\e\partial_k\dot p):  \partial_k \xi\, \varphi^6+\frac{1}{2C_*} \iint_{Q_t}(\partial_k \sigma-\e\partial_k\dot p):(\partial_k \sigma-\e\partial_k\dot p)\varphi^6\\
\leq C\left(\frac12 \iint_{Q_t}  \partial_k  \sigma :\partial_k  e\, \varphi^6+ \|\dot u\|^2_{L^2(Q_T;\R^2)} +\|f\|^2_{L^2(0,T;H^2(\O;\R^2))} \right)\\
- 2 \iint_{Q_t}\partial_k  \sigma_{ij}  (E_{kj} \dot u) \partial_i \varphi^6.
\end{multline}
Using again \eqref{eq:useful-form}, we have
\begin{multline*}
\iint_{Q_t} \partial_k  \sigma_{ij} (E_{kj} \dot u) \partial_i \varphi^6=\iint_{Q_t} \dot e_{kj} \, \partial_k  \sigma_{ij} \, \partial_i \varphi^6
+ \e \iint_{Q_t} \dot p_{kj} \,  \partial_k \dot p_{ij}\, \partial_i \varphi^6\\
 + \iint_{Q_t} (\partial_k  \sigma_{ij} -\e \partial_k \dot p_{ij})\left(\xi_{kj} -e_{kj}+\e(\CC^{-1}\dot p)_{kj} + \frac{\kappa}{4\alpha K_0}\delta_{kj}\right) \partial_i \varphi^6.
\end{multline*}
Applying again the Cauchy-Schwarz and Young inequalities, we get that
\begin{multline*}
-2\iint_{Q_t} \partial_k  \sigma_{ij} (E_{kj} \dot u) \partial_i \varphi^6 \leq -2\iint_{Q_t} (\partial_k  \sigma_{ij} -\e \partial_k \dot p_{ij})\xi_{kj} \, \partial_i\varphi^6\\
+C \Bigg(1+\frac12\iint_{Q_t}  \partial_k  \sigma : \partial_k  e\, \varphi^6 + \|e\|^2_{H^1([0,T];L^2(\O;\Ms))} 
+ \e \|\dot p\|^2_{L^2(Q_T;\Ms)}\Bigg) \\
 + \frac{\e}{4} \iint_{Q_t} \partial_k \dot p: \partial_k \dot p\,  \varphi^6+\frac{1}{4C_*}\iint_{Q_t}(\partial_k \sigma-\e\partial_k\dot p):(\partial_k \sigma-\e\partial_k\dot p)\varphi^6.
\end{multline*}
Inserting inside \eqref{eq:stim111} yields
\begin{multline}\label{eq:stim2}
\iint_{Q_t}  \partial_k \sigma : \partial_k \dot e\, \varphi^6 + \frac{\e}{4} \iint_{Q_t} \partial_k \dot p: \partial_k \dot p\,  \varphi^6\\
+\frac12 \iint_{Q_t}  (\partial_k \sigma-\e\partial_k\dot p): \partial_k \xi\, \varphi^6+\frac{1}{4C_*}\iint_{Q_t}(\partial_k \sigma-\e\partial_k\dot p):(\partial_k \sigma-\e\partial_k\dot p)\varphi^6\\
\leq - 2\iint_{Q_t} (\partial_k  \sigma_{ij} -\e \partial_k \dot p_{ij})\xi_{kj}\, \partial_i\varphi^6+ C\Bigg(1+\frac12 \iint_{Q_t}  \partial_k  \sigma : \partial_k  e\, \varphi^6 + \|\dot u\|^2_{L^2(Q_T;\R^2)}\\
+ \|e\|^2_{H^1([0,T];L^2(\O;\Ms))} + \e \|\dot p\|^2_{L^2(Q_T;\Ms)} +\|f\|^2_{L^2(0,T;H^2(\O;\R^2))}\Bigg).
\end{multline}
The remaining of the proof consists in absorbing the integral
$$\iint_{Q_t} (\partial_k  \sigma_{ij} -\e \partial_k \dot p_{ij})\xi_{kj}\, \partial_i\varphi^6$$
 by the left hand side of \eqref{eq:stim2}.

\medskip

{\bf Step 3.} We consider the following decomposition:
$$\iint_{Q_t} (\partial_k  \sigma_{ij}-\e \partial_k \dot p_{ij}) \xi_{kj}\, \partial_i \varphi^6=\sum_{r=1}^3 \iint_{A_r^\xi \cap Q_t} (\partial_k  \sigma_{ij}-\e \partial_k \dot p_{ij}) \xi_{kj}\, \partial_i \varphi^6.$$
We first observe that, according to Corollary~\ref{cor:form-sigma}, $\partial_k \sigma-\e\partial_k\dot p=0$ a.e.\ in $A_2^\xi$, so that 
\begin{equation}\label{eq:Omega2}
\iint_{A_2^\xi \cap Q_t} (\partial_k  \sigma_{ij}-\e \partial_k \dot p_{ij}) \xi_{kj}\, \partial_i \varphi^6=0.
\end{equation}
For $r=1$, using that $\partial_i \varphi^6=\varphi^3 (6 \varphi^2 \partial_i \varphi)$, together with the Cauchy-Schwarz and Young inequalities, yields
\begin{multline}\label{eq:Omega1}
\left|\iint_{A_1^\xi \cap Q_t} (\partial_k  \sigma_{ij}-\e \partial_k \dot p_{ij}) \xi_{kj}\, \partial_i \varphi^6\right|\\
 \leq \frac{1}{8C_*} \iint_{Q_t}(\partial_k \sigma-\e\partial_k\dot p):(\partial_k \sigma-\e\partial_k\dot p)\varphi^6 + C\|\xi\|^2_{L^2(A_1^\xi;\Ms)}.
\end{multline}

By formula \eqref{form-sigma} we have that
$$\sigma-\e\dot p=Dg\left(\xi + \frac{\kappa}{4\alpha K_0}{\rm Id}\right).$$
By convexity of $g$ this implies that
$$0= g(0)  \geq g\left(\xi + \frac{\kappa}{4\alpha K_0}{\rm Id}\right) - (\sigma-\e\dot p):\left(\xi + \frac{\kappa}{4\alpha K_0}{\rm Id}\right)$$
a.e.\ in $\O \times (0,T)$.
Integrating the previous inequality over $A_1^\xi$ and using the coercivity property \eqref{eq:coerc-g} of $g$ yield
$$a_0\iint_{A_1^\xi}\left|\xi + \frac{\kappa}{4\alpha K_0}{\rm Id}\right|^2  \leq \iint_{A_1^\xi}g\left(\xi + \frac{\kappa}{4\alpha K_0}{\rm Id}\right)  \leq \iint_{A_1^\xi}(\sigma-\e\dot p):\left(\xi + \frac{\kappa}{4\alpha K_0}{\rm Id}\right).$$
Owing to the Cauchy-Schwarz inequality, we deduce that 
$$\iint_{A_1^\xi}\left|\xi + \frac{\kappa}{4\alpha K_0}{\rm Id}\right|^2 \leq \frac{1}{a_0^2} \iint_{Q_T} |\sigma-\e\dot p|^2,$$
and therefore,
\begin{equation}\label{eq:estim-xi}
\|\xi\|^2_{L^2(A_1^\xi;\Ms)} \leq C \left(1+\|e\|^2_{L^2(Q_T;\Ms)} +\e\|\dot p\|^2_{L^2(Q_T;\Ms)} \right).
\end{equation}

Inserting \eqref{eq:Omega2}, \eqref{eq:Omega1}, and \eqref{eq:estim-xi} inside \eqref{eq:stim2}, leads to 
\begin{multline*}
\iint_{Q_t}  \partial_k \sigma : \partial_k \dot e\, \varphi^6 + \frac{\e}{4}\iint_{Q_t} \partial_k \dot p: \partial_k \dot p\,  \varphi^6\\
+\frac12\iint_{Q_t}  (\partial_k \sigma-\e\partial_k\dot p): \partial_k \xi\, \varphi^6+\frac{1}{8C_*}\iint_{Q_t}(\partial_k \sigma-\e\partial_k\dot p):(\partial_k \sigma-\e\partial_k\dot p)\varphi^6\\
\leq  C\Bigg(1+\frac12 \iint_{Q_t}  \partial_k  \sigma : \partial_k  e\, \varphi^6 + \|\dot u\|^2_{L^2(Q_T;\R^2)}
 + \|e\|^2_{H^1([0,T];L^2(\O;\Ms))}\\ + \e \|\dot p\|^2_{L^2(Q_T;\Ms)} +\|f\|^2_{L^2(0,T;H^2(\O;\R^2))} +\sum_{i,k=1}^2 \iint_{ A_3^\xi\cap Q_t} |\partial_k  \sigma -\e \partial_k \dot p| |\xi| |\partial_i\varphi^6|\Bigg).
\end{multline*}

According to \eqref{eq:xiA3}, and by definition of the set $A_3^\xi$, we have $|\xi | \leq c|\xi_D|$ a.e.\ in $A_3^\xi$, hence
\begin{multline}\label{eq:stim6}
\iint_{Q_t}  \partial_k \sigma : \partial_k \dot e\, \varphi^6 + \frac{\e}{4}\iint_{Q_t}  \partial_k \dot p: \partial_k \dot p\,  \varphi^6\\
+\frac12\iint_{Q_t}  (\partial_k \sigma-\e\partial_k\dot p): \partial_k \xi\, \varphi^6+\frac{1}{8C_*}\iint_{Q_t}(\partial_k \sigma-\e\partial_k\dot p):(\partial_k \sigma-\e\partial_k\dot p)\varphi^6 \\
\leq  C\Bigg(1+\frac12 \iint_{Q_t}  \partial_k  \sigma : \partial_k  e\, \varphi^6 + \|\dot u\|^2_{L^2(Q_T;\R^2)}
 + \|e\|^2_{H^1([0,T];L^2(\O;\Ms))}\\ + \e \|\dot p\|^2_{L^2(Q_T;\Ms)} +\|f\|^2_{L^2(0,T;H^2(\O;\R^2))}+\sum_{i,k=1}^2 \iint_{A_3^\xi \cap Q_t} |\partial_k  \sigma -\e \partial_k \dot p| |\xi_D| |\partial_i\varphi^6|\Bigg).
\end{multline}

\medskip

{\bf Step 4.} We next focus on the last term
$$\sum_{i,k=1}^2 \iint_{A_3^\xi \cap Q_t} |\partial_k  \sigma -\e \partial_k \dot p| |\xi_D| |\partial_i\varphi^6|$$
in the right hand side of \eqref{eq:stim6}. We will show that this term can be absorbed by the left hand side of  \eqref{eq:stim6} by using the explicit expression \eqref{form-sigma} of $\sigma-\e\dot p$. Let us consider the functions
$$h:=2\Dc\left(|\xi_D|-\frac{{\rm tr} \xi }{2\alpha}  \right), \quad B:= \frac{\xi_D}{|\xi_D|}-\frac{1}{2\alpha}{\rm Id} .$$
According to \eqref{form-sigma} and \eqref{eq:G}, we have
\begin{equation}\label{eq:hat-sigma}
 \sigma-\e\dot p = \frac{\kappa}{2\alpha} {\rm Id} + h B \quad  \text{a.e.\ in } A_3^\xi,
\end{equation}
and by Corollary~\ref{cor:form-sigma} we obtain
\begin{equation}\label{eq:dk-hat-sigma}
\partial_k  \sigma-\e\partial_k\dot p=(\partial_k h) B+ \frac{h}{| \xi_D|}\left(\partial_k \xi_D- \frac{\xi_D : \partial_k\xi_D}{|\xi_D|^2}\xi_D \right)\hfill \text{ a.e.\ in }  A^{\xi}_3.
\end{equation}
As a consequence, since $|B|\leq 1+\alpha^{-1}$, we deduce that
\begin{multline}\label{eq:avant-derniere-est}
\iint_{A_3^\xi \cap Q_t} |\partial_k  \sigma-\e\partial_k\dot p| |\xi_D|  |\nabla \varphi^6|\\
\leq (1+\alpha^{-1}) \iint_{A_3^\xi \cap Q_t} |\nabla h| |\xi_D|  |\nabla \varphi^6| + \iint_{A_3^\xi \cap Q_t} h\left| \partial_k \xi_D- \frac{\xi_D : \partial_k\xi_D}{|\xi_D|^2}\xi_D \right|  |\nabla \varphi^6|.
\end{multline}
We start by estimating the first integral in the right hand side of \eqref{eq:avant-derniere-est}. Let us show that there exists a constant $c_\alpha>0$ (only depending on $\alpha$) such that
\begin{equation}\label{nabla h}
|\nabla h|\leq c_\alpha \e (\partial_k \dot p : \partial_k \dot p)^{1/2}
+\frac{c_\alpha h}{|\xi_D|}\left[ \sum_{l=1}^2\left( | \partial_l \xi_D|^2 - \frac{(\xi_D : \partial_l \xi_D)^2}{|\xi_D|^2}\right) \right]^{1/2}  + c_\alpha |f|
\quad \text{ a.e.\ in }  A^{\xi}_3.
\end{equation}
Writing componentwise the equilibrium equation \eqref{eq:div} yields  $-\partial_j  \sigma_{ij}=f_i$ in $\O \times (0,T)$, and using \eqref{eq:dk-hat-sigma},
$$ \e\partial_j \dot p_{ij}+ (\partial_j h) B_{ij}+ \frac{h}{| \xi_D|}\left(\partial_j (\xi_D)_{ij}- \frac{\xi_D : \partial_j\xi_D}{|\xi_D|^2}(\xi_D)_{ij}\right)  +f_i=0\quad \text{ a.e.\ in }  A^{\xi}_3.$$
According to \cite[Lemma~3.1]{Se}, the matrix $B$ is invertible if and only if $\alpha \neq 1/\sqrt 2$, and in that case, $|B^{-1}|\leq c_\alpha$ for some constant $c_\alpha>0$. Hence
$$\partial_i h=-\e (B^{-1})_{ik}\,\partial_j \dot p_{kj}
-\frac{h}{|\xi_D|}(B^{-1})_{ik}\left(\partial_j (\xi_D)_{kj}- \frac{\xi_D : \partial_j\xi_D}{|\xi_D|^2}(\xi_D)_{kj}\right) -(B^{-1})_{ik}f_k \quad \text{ a.e.\ in }  A^{\xi}_3,$$
and
\begin{multline*}
|\nabla h| \leq 
c_\alpha \e (\partial_k \dot p : \partial_k \dot p)^{1/2}
+ \frac{c_\alpha h}{|\xi_D|}\left[\sum_{j,k,l=1}^2\left| \partial_l (\xi_D)_{kj}- \frac{\xi_D : \partial_l\xi_D}{|\xi_D|^2}(\xi_D)_{kj}\right|^2\right]^{1/2}  + c_\alpha |f|
\\
= c_\alpha \e (\partial_k \dot p : \partial_k \dot p)^{1/2}
+ \frac{c_\alpha h}{|\xi_D|}\left[ \sum_{l=1}^2 \left(| \partial_l \xi_D|^2 - \frac{(\xi_D : \partial_l \xi_D)^2}{|\xi_D|^2}\right) \right]^{1/2}  + c_\alpha |f|
\quad \text{ a.e.\ in }  A^{\xi}_3,
 \end{multline*}
which proves \eqref{nabla h}. According to this estimate, the fact that  $\nabla \varphi^6=\varphi^3 (6 \varphi^2 \nabla \varphi)$, and the Cauchy-Schwarz inequality, we get that
\begin{eqnarray}
\lefteqn{\iint_{A_3^\xi \cap Q_t}  |\nabla h| |\xi_D|  |\nabla \varphi^6|  \leq 
c_\alpha\e \iint_{A_3^\xi \cap Q_t}  (\partial_k \dot p : \partial_k \dot p)^{1/2} |\xi_D| |\nabla \varphi^6|}
\nonumber \\
& & {}+c_\alpha\iint_{A_3^\xi \cap Q_t}  h\left[ \sum_{l=1}^2\left( | \partial_l \xi_D|^2 - \frac{(\xi_D : \partial_l \xi_D)^2}{|\xi_D|^2} \right)\right]^{1/2}|\nabla \varphi^6| +c_\alpha \iint_{Q_t}|f| |\xi_D| |\nabla \varphi^6|
\nonumber \\ 
&\leq &
6 c_\alpha \e \left(\iint_{A_3^\xi \cap Q_t}  \partial_k \dot p : \partial_k \dot p \, \varphi^6\right)^{1/2}
\left(\iint_{A_3^\xi \cap Q_t}  |\xi_D|^2\varphi^4|\nabla \varphi|^2\right)^{1/2}
\nonumber \\
& &+ 6 c_\alpha \left(\iint_{A_3^\xi \cap Q_t}  \frac{h}{|\xi_D|} \left[ \sum_{l=1}^2 \left(| \partial_l \xi_D|^2 - \frac{(\xi_D : \partial_l \xi_D)^2}{|\xi_D|^2}\right) \right] \varphi^6\right)^{1/2}\left(\iint_{A_3^\xi \cap Q_t}  h|\xi_D|\varphi^4|\nabla \varphi|^2\right)^{1/2}
\nonumber \\
& &  + 6c_\alpha \|f\|_{L^\infty(Q_T;\R^2)} \|\xi_D\|_{L^1(Q_T;\Ms)}\nonumber\\
& \leq & 
6 c_\alpha \left( \e \iint_{Q_t} \partial_k \dot p : \partial_k \dot p \, \varphi^6\right)^{1/2}
\left( \e\iint_{Q_t} |\xi_D|^2\varphi^4|\nabla \varphi|^2\right)^{1/2}
\nonumber \\
&& 
{}+ 6 c_\alpha \left(\iint_{A_3^\xi \cap Q_t}  (\partial_k  \sigma-\e\partial_k\dot p): \partial_k\xi\, \varphi^6 \right)^{1/2}\left(\iint_{A_3^\xi \cap Q_t}  h|\xi_D|\varphi^4|\nabla \varphi|^2\right)^{1/2}\nonumber\\
& & {} + 6c_\alpha \|f\|_{L^\infty(Q_T;\R^2)} \|\xi_D\|_{L^1(Q_T;\Ms)},
\label{eq:avant-derniere-est2}
\end{eqnarray}
where we used \eqref{eq:sigmakxik} in the last inequality. 

We next estimate the second integral in the right hand side of \eqref{eq:avant-derniere-est}. Using again that $\nabla \varphi^6=\varphi^3 (6 \varphi^2 \nabla \varphi)$, together with the Cauchy-Schwarz inequality, we infer that
\begin{multline}\label{eq:avant-derniere-est3}
\sum_{k=1}^2\iint_{A_3^\xi \cap Q_t}  h\left| \partial_k \xi_D- \frac{\xi_D : \partial_k\xi_D}{|\xi_D|^2}\xi_D \right|  |\nabla \varphi^6|\\
 \leq 6 \left(\iint_{A_3^\xi \cap Q_t}  \frac{h}{|\xi_D|} \left[ \sum_{k=1}^2 \left(| \partial_k \xi_D|^2 - \frac{(\xi_D : \partial_k \xi_D)^2}{|\xi_D|^2}\right) \right] \varphi^6 \right)^{1/2}\left(\iint_{A_3^\xi \cap Q_t}  h|\xi_D|\varphi^4|\nabla \varphi|^2\right)^{1/2}\\
\leq  6\left(\iint_{A_3^\xi \cap Q_t}  (\partial_k  \sigma-\e\partial_k\dot p): \partial_k\xi\, \varphi^6\right)^{1/2}\left(\iint_{A_3^\xi \cap Q_t}  h|\xi_D|\varphi^4|\nabla \varphi|^2\right)^{1/2},
\end{multline}
where, once more, we used \eqref{eq:sigmakxik} in the last inequality. Gathering \eqref{eq:avant-derniere-est}, \eqref{eq:avant-derniere-est2}, \eqref{eq:avant-derniere-est3}, inserting inside \eqref{eq:stim6}, and using Young's inequality yield
\begin{multline}\label{eq:stim7bis}
\iint_{Q_t}  \partial_k \sigma : \partial_k \dot e\, \varphi^6 + \frac{\e}{8} \iint_{Q_t}   \partial_k \dot p:  \partial_k \dot p\, \varphi^6\\
+\frac14 \iint_{Q_t}  (\partial_k \sigma-\e\partial_k\dot p): \partial_k \xi\, \varphi^6+\frac{1}{8C_*}\iint_{Q_t} (\partial_k \sigma-\e\partial_k\dot p):(\partial_k \sigma-\e\partial_k\dot p)\varphi^6 \\
\leq  C\Bigg(1+\frac12\iint_{Q_t}  \partial_k  \sigma: \partial_k  e\, \varphi^6+ \|\dot u\|^2_{L^2(Q_T;\R^2)} + \|e\|^2_{H^1([0,T];L^2(\O;\Ms))}\\
 + \e \|\dot p\|^2_{L^2(Q_T;\Ms)}+\e \|\xi_D\|^2_{L^2(Q_T;\Ms)}+ \|\xi_D\|^2_{L^1(Q_T;\Ms)}\\
 +\|f\|^2_{L^2(0,T;H^2(\O;\R^2))} +\|f\|^2_{L^\infty(Q_T;\R^2)}+ \iint_{A_3^\xi \cap Q_t}  h|\xi_D|\varphi^4|\nabla \varphi|^2\Bigg).
\end{multline}
By the definition of $\xi$ we obtain
\begin{eqnarray}
\sqrt \e \|\xi_D\|_{L^2(Q_T;\Ms)} & \leq & C\sqrt\e\|\dot p\|_{L^2(Q_T;\Ms)}+ \|e\|_{L^2(Q_T;\Ms))},\label{est xiD}\\
\|\xi_D\|_{L^1(Q_T;\Ms)} & \leq & C\left(\|\dot p\|_{L^1(Q_T;\Ms)}+ \|e\|_{L^2(Q_T;\Ms))}\right).\nonumber
\end{eqnarray}
Thus, equation \eqref{eq:stim7bis} can be rewritten as
\begin{multline}\label{eq:stim7}
\iint_{Q_t}  \partial_k \sigma : \partial_k \dot e\, \varphi^6 + \frac{\e}{8} \iint_{Q_t}   \partial_k \dot p:  \partial_k \dot p\, \varphi^6\\
+\frac14 \iint_{Q_t}  (\partial_k \sigma-\e\partial_k\dot p): \partial_k \xi\, \varphi^6+\frac{1}{8C_*}\iint_{Q_t} (\partial_k \sigma-\e\partial_k\dot p):(\partial_k \sigma-\e\partial_k\dot p)\varphi^6 \\
\leq  C\Bigg(1+\frac12\iint_{Q_t}  \partial_k  \sigma: \partial_k  e\, \varphi^6 + \|\dot u\|^2_{L^2(Q_T;\R^2)} + \|e\|^2_{H^1([0,T];L^2(\O;\Ms))}\\
 + \e \|\dot p\|^2_{L^2(Q_T;\Ms)} +\|\dot p\|^2_{L^1(Q_T;\Ms)} +\|f\|^2_{L^2(0,T;H^2(\O;\R^2))} +\|f\|^2_{L^\infty(Q_T;\R^2)}\\
+ \iint_{A_3^\xi \cap Q_t}  h|\xi_D|\varphi^4|\nabla \varphi|^2\Bigg).
\end{multline}

\medskip

{\bf Step 5.} We now estimate the last term
$$\iint_{A_3^\xi \cap Q_t}  h|\xi_D|\varphi^4|\nabla \varphi|^2$$
in the right hand side of \eqref{eq:stim7}. To simplify notation, let us denote by $\psi:=\varphi^4 |\nabla \varphi|^2$. Firstly, according to \eqref{eq:hat-sigma}, we have that $ \sigma_D-\e\dot p_D=h \xi_D/|\xi_D|$ a.e.\ in $A_3^\xi$, and thus $h|\xi_D|=(\sigma_D-\e\dot p_D):\xi_D$ a.e.\ in $A_3^\xi$. This implies that
\begin{equation}\label{eq:estim71}
\iint_{A_3^\xi \cap Q_t}  h|\xi_D|\psi =\iint_{A_3^\xi \cap Q_t}   (\sigma_D-\e \dot p_D) :\xi_D \psi.
\end{equation}
According to \eqref{form-sigma} and \eqref{eq:G}, we have that $\sigma_D-\e \dot p_D=0$ a.e.\ in $A_2^\xi$, and thus the Cauchy-Schwarz inequality yields
$$\iint_{A_3^\xi \cap Q_t}  h|\xi_D|\psi  \leq \iint_{Q_t}  (\sigma_D-\e \dot p_D) :\xi_D\psi+C\|\sigma-\e \dot p\|_{L^2(Q_T;\Ms)} \|\xi\|_{L^2(A_1^\xi;\Ms)}.$$
On the other hand, the definition of $\xi$ and the additive decomposition \eqref{eq:useful-form} yield
\begin{multline}\label{eq:estim72}
\iint_{Q_t}   (\sigma_D-\e \dot p_D) :\xi_D\psi\leq\iint_{Q_t}    \sigma :(E\dot u)\psi - \frac{1}{2}\iint_{Q_t}  ( {\rm tr}\sigma) ({\rm div}\dot u)\psi\\
+\|\sigma\|_{L^2(Q_T:\Ms)}\left(\|e\|_{H^1([0,T];L^2(\O:\Ms))}+ \e \|\dot p\|_{L^2(Q_T;\Ms)} \right)\\
+ \e \|\dot p\|_{L^2(Q_T;\Ms)} \|\xi_D\|_{L^2(Q_T;\Ms)}.
\end{multline}
Inserting \eqref{eq:estim71} and \eqref{eq:estim72} into \eqref{eq:stim7}, and using \eqref{est xiD} lead to 
\begin{multline}\label{eq:stim8}
\iint_{Q_t}  \partial_k \sigma: \partial_k \dot e \, \varphi^6 + \frac{\e}{8}\iint_{Q_t} \partial_k \dot p: \partial_k \dot p\, \varphi^6\\
+\frac14\iint_{Q_t}    (\partial_k \sigma-\e\partial_k\dot p): \partial_k \xi\, \varphi^6+\frac{1}{8C_*} \iint_{Q_t}  (\partial_k \sigma-\e\partial_k\dot p):(\partial_k \sigma-\e\partial_k\dot p)\varphi^6\\
\leq  C\Bigg(1+\frac12\iint_{Q_t}  \partial_k  \sigma:\partial_k  e\, \varphi^6 + \|\dot u\|^2_{L^2(Q_T;\R^2)}
 + \|e\|^2_{H^1([0,T];L^2(\O;\Ms))}\\
 + \e \|\dot p\|^2_{L^2(Q_T;\Ms)} +\|\dot p\|^2_{L^1(Q_T;\Ms)} +\|f\|^2_{L^2(0,T;H^2(\O;\R^2))} +\|f\|^2_{L^\infty(Q_T;\R^2)}\\
 +\iint_{Q_t}    \sigma :(E\dot u)\psi  - \frac{1}{2}\iint_{Q_t}  ( {\rm tr}\sigma) ({\rm div}\dot u)\psi\Bigg).
\end{multline}

{\bf Step 6.} We finally estimate both last integrals in the right hand side of \eqref{eq:stim8}. Integrating by parts, and using again the Cauchy-Schwarz and Young inequalities yield
\begin{multline*}
\iint_{Q_t}  \sigma :(E\dot u)\psi  - \frac{1}{2}\iint_{Q_t}( {\rm tr}\sigma) ({\rm div}\dot u)\psi
= -\iint_{Q_t} ({\rm div }\sigma) \cdot \dot u \psi -\iint_{Q_t} \sigma : (\dot u \otimes \nabla \psi)\\
 + \frac{1}{2} \iint_{Q_t} \nabla ( {\rm tr}\sigma) \cdot\dot u\psi +\frac{1}{2} \iint_{Q_t} {\rm tr}\sigma (\dot u \cdot \nabla \psi)\\
 \leq C\left(\frac12\iint_{Q_t}  \partial_k  \sigma : \partial_k  e\, \varphi^6+\|e\|^2_{L^2(Q_T;\Ms)}+\|\dot  u\|^2_{L^2(Q_T;\Ms)}\right).
\end{multline*}
Inserting this ultimate relation inside \eqref{eq:stim8} leads to
\begin{multline}\label{eq:stim9}
\iint_{Q_t}  \partial_k \sigma : \partial_k \dot e\, \varphi^6+ \frac{\e}{8} \iint_{Q_t} \partial_k \dot p : \partial_k \dot p\, \varphi^6\\
+\frac14 \iint_{Q_t}  (\partial_k \sigma-\e\partial_k\dot p): \partial_k \xi\, \varphi^6+\frac{1}{8C_*} \iint_{Q_t}(\partial_k \sigma-\e\partial_k\dot p):(\partial_k \sigma-\e\partial_k\dot p)\varphi^6\\
\leq  C\bigg(1+\frac12\iint_{Q_t}  \partial_k  \sigma: \partial_k  e\, \varphi^6 + \|\dot u\|^2_{L^2(Q_T;\R^2)} + \|e\|^2_{H^1([0,T];L^2(\O;\Ms))} + \e \|\dot p\|^2_{L^2(Q_T;\Ms)}\\
+\|\dot p\|^2_{L^1(Q_T;\Ms)} +\|f\|^2_{L^2(0,T;H^2(\O;\R^2))} +\|f\|^2_{L^\infty(Q_T;\R^2)}\bigg).
\end{multline}

{\bf Step 7.} The final step rests on an application of Gronwall's Lemma. For all $t \in [0,T]$ let us denote 
$$q(t):=\frac12\int_{\O} \partial_k  \sigma(t):\partial_k  e(t)\, \varphi^6\, dx.$$
Neglecting nonnegative terms at the left hand side of \eqref{eq:stim9} (note that the third integral at the left hand side of \eqref{eq:stim9}
is nonnegative by \eqref{eq:sigmakxik}), we can write \eqref{eq:stim9} as
$$q(t) \leq  q(0)+ a \left(\int_0^t q(s)\, ds + b\right),$$
where $a:=C$ and
\begin{multline*}b:=1+\|\dot u\|^2_{L^2(Q_T;\R^2)} + \|e\|^2_{H^1([0,T];L^2(\O;\Ms))} + \e \|\dot p\|^2_{L^2(Q_T;\Ms)}\\
 +\|\dot p\|^2_{L^1(Q_T;\Ms)} +\|f\|^2_{L^2(0,T;H^2(\O;\R^2))} +\|f\|^2_{L^\infty(Q_T;\R^2)}.
\end{multline*}
By applying Gronwall's inequality we deduce that
$$q(t) \leq (q(0) + ab)e^{at} \quad \text{ for all }t \in [0,T].$$
Using \eqref{eq:coerc-Q}, together with the fact that $\varphi=1$ on $\O'$,  $\varphi=0$ on $\O\setminus \overline{\O''}$ and $\varphi\geq 0$, 
the previous inequality leads to
\begin{multline}\label{eq:endly}
\int_{\O'}  \partial_k  \sigma(t):\partial_k  \sigma(t)\, dx \leq  C\left( 1+ \|e_0\|^2_{H^1(\O'';\Ms)} +\|\dot u\|^2_{L^2(Q_T;\R^2)}+ \|e\|^2_{H^1([0,T];L^2(\O;\Ms))}\right.\\
\left.  + \e \|\dot p\|^2_{L^2(Q_T;\Ms)}+\|\dot p\|^2_{L^1(Q_T;\Ms)}+\|f\|^2_{L^2(0,T;H^2(\O;\R^2))} +\|f\|^2_{L^\infty(Q_T;\R^2)}\right).
\end{multline}
Using this inequality, the {\it a priori} estimates \eqref{eq:first-estim}, \eqref{eq:second-estim}, and \eqref{eq:third-estim-bis}, and the continuous embedding of $BD(\O)$ into $L^2(\O;\R^2)$, we conclude that the right hand side of \eqref{eq:endly} is uniformly bounded with respect to $\e$.
\end{proof}

We are now in position to complete the proof of our regularity result.

\begin{proof}[Proof of Theorem~\ref{thm:main}]
According to Propositions~\ref{prop:perzyna}, \ref{prop:convergence}, and \ref{prop:unif-bound}, the solution $\sigma_\e$ to the viscoplastic model is uniformly bounded in $L^\infty(0,T;H^1_{\rm loc}(\O;\Ms))$ and converges to $\sigma$ weakly in $H^1([0,T];L^2(\O;\Ms))$,
where $\sigma$ is the solution of the perfectly plastic model. By uniqueness of weak limits, we infer that $\sigma \in L^\infty(0,T;H^1_{\rm loc}(\O;\Ms))$.
\end{proof}

\section{Strong form of the flow rule}

In this section we use the Sobolev regularity property established in Theorem~\ref{thm:main}  
to define a suitable intrisic representative of the stress, for which the flow rule holds in a pointwise form, and not only in the measure theoretical sense of \eqref{eq:fr} (see \cite{A2,DMDSM,FGM} for analogous results). To this aim we define the averages of the stress as follows: for all $(x,t) \in \overline \O \times [0,T]$, and all $r>0$, let 
$$\sigma_r(x,t):=\frac{1}{\LL^2(B_r(x) \cap \O)} \int_{B_r(x) \cap \O}\sigma(y,t)\, dy.$$
From Theorem~\ref{thm:main}, we know that $\sigma(t) \in H^1_{\rm loc}(\O;\Ms)$ for a.e.\ $t \in [0,T]$. Therefore, for a.e.\ $t \in [0,T]$ there exists a quasi-continuous representative of $\sigma(t)$, denoted by $\tilde \sigma(t)$, such that
$$\sigma_r(t) \to \tilde \sigma(t) \quad \text{ in }\O \setminus Z_t,$$
where $Z_t \subset \O$ is a Borel set of zero capacity in $\O$. 

\begin{thm}\label{thm:strong-form}
For a.e.\ $t \in [0,T]$ the quasi-continuous representative of the stress $\tilde \sigma(t)$ is $|\dot p(t)|$-measurable and 
$$H\left(\frac{d\dot p(t)}{d|\dot p(t)|} \right) = \tilde \sigma(t) : \frac{d\dot p(t)}{d|\dot p(t)|} \quad |\dot p(t)|\text{-a.e.\ in }\O.$$
\end{thm}

\begin{proof}

Since for a.e.\ $t \in [0,T]$, $\sigma(t) \in L^2(\O;\Ms)$ with $\div \sigma(t) \in L^2(\O;\R^2)$, we have that
\begin{eqnarray*}
&&\sigma_r(t) \to \sigma(t)  \text{ in }  L^2(\O;\Ms),\\
&& \div \sigma_r(t) \to \div \sigma(t)  \text{ in }  L^2(\O;\R^2),
\end{eqnarray*}
as $r \to 0^+$. As a consequence of the definition \eqref{eq:duality} of the stress/strain duality, we infer that, 
\begin{equation}\label{eq:Dprime}
[ \sigma_r(t):\dot p(t)] \wto[ \sigma(t):\dot p(t)] \quad \text{ weakly* in }\mathcal D'(\R^2).
\end{equation}

Since $\Cap(Z_t)=0$, we have $\HH^1(Z_t)=0$. Thus, since $\dot u(t) \in BD(\O)$, by \cite[Remark 3.3]{ACDM} we get that $|E\dot u(t)|(Z_t)=0$, and thus $|\dot p(t)|(Z_t)=0$. It thus follows that for a.e.\ $t \in [0,T]$,
$$\sigma_r(t) \to \tilde \sigma(t) \quad |\dot p(t)|\text{-a.e.\ in }\O,$$
and since $\sigma_r(t)\in \mathcal C(\overline \O;\Ms)$, it results that $\tilde \sigma(t)$ is $|\dot p(t)|$-measurable. Using next that $\sigma_r(x,t) \in K$ for all $t \in [0,T]$ and all $x \in \O$, we get that
$$H\left(\frac{d\dot p(t)}{d|\dot p(t)|}(x) \right) \geq \sigma_r(x,t) : \frac{d\dot p(t)}{d|\dot p(t)|}(x) \quad \text{ for $|\dot p(t)|$-a.e.\ $x \in \O$,}$$
and, passing to the limit as $r \to 0^+$,
\begin{equation}\label{eq:Hgeq2}
H\left(\frac{d\dot p(t)}{d|\dot p(t)|}(x) \right) \geq \tilde \sigma(x,t) : \frac{d\dot p(t)}{d|\dot p(t)|}(x) \quad \text{ for $|\dot p(t)|$-a.e.\ $x \in \O$.}
\end{equation}
According to Fatou's Lemma, we infer that for all $\varphi \in \C^\infty_c(\O)$ with $\varphi \geq 0$,
\begin{multline*}
\int_\O \left[H\left(\frac{d\dot p(t)}{d|\dot p(t)|} \right) - \tilde \sigma(t) : \frac{d\dot p(t)}{d|\dot p(t)|} \right] \varphi\,  d|\dot p(t)|\\
=\int_\O \liminf_{r \to 0^+}\left[H\left(\frac{d\dot p(t)}{d|\dot p(t)|} \right) - \sigma_r(t) : \frac{d\dot p(t)}{d|\dot p(t)|} \right]\varphi\, d|\dot p(t)|\\
\leq \liminf_{r \to 0^+}\int_\O \left[H\left(\frac{d\dot p(t)}{d|\dot p(t)|} \right) - \sigma_r(t) : \frac{d\dot p(t)}{d|\dot p(t)|} \right]\varphi\, d|\dot p(t)|.
\end{multline*}
Since $\sigma_r(t) \in  \mathcal C(\overline \O;\Ms) \cap \mathcal S$, we can use Remark \ref{product} together with the convergence \eqref{eq:Dprime} to get that
\begin{multline}\label{eq:1030}
\int_\O \left[H\left(\frac{d\dot p(t)}{d|\dot p(t)|} \right) - \tilde \sigma(t) : \frac{d\dot p(t)}{d|\dot p(t)|} \right] \varphi\,  d|\dot p(t)|\\
\leq \int_\O \varphi \, d H(\dot p(t)) -  \limsup_{r \to 0^+}\langle[ \sigma_r(t):\dot p(t)],\varphi\rangle \\
=  \int_\O \varphi \, d H(\dot p(t))-\langle[ \sigma(t):\dot p(t)] ,\varphi\rangle=0,
\end{multline}
where we used the flow rule \eqref{eq:fr} in the last equality. Finally, gathering \eqref{eq:Hgeq2} and \eqref{eq:1030}, we obtain that for a.e.\ $t \in [0,T]$,
$$H\left(\frac{d\dot p(t)}{d|\dot p(t)|} \right) = \tilde \sigma(t) : \frac{d\dot p(t)}{d|\dot p(t)|} \quad |\dot p(t)|\text{-a.e.\ in }\O,$$
as required.
\end{proof}

\section*{Acknowledgements}

\noindent The authors would like to thank the hospitality of SISSA, where a large part of this work has been carried out. The research has been supported by the ERC under Grant No.~290888 ``Quasistatic and Dynamic Evolution Problems in Plasticity and Fracture'', 
by GNAMPA under Project 2016 ``Multiscale analysis of complex systems with variational methods",
and by the International Associate Laboratory, Laboratorio Ypazia delle Scienze Matematiche (LYSM), between CNRS and INdAM.

\end{document}